\documentclass[11pt, twoside, a4paper]{amsart}

\usepackage[utf8]{inputenc}
\usepackage[english]{babel}
\usepackage{amsmath} 
\usepackage{amssymb}
\usepackage{amsthm}
\usepackage{mathtools}
\usepackage{graphicx}
\graphicspath{ {c:/This PC/Documents/images/} }
\usepackage{tikz-cd}
\usepackage{MnSymbol}
\usepackage{amsthm}

\usepackage[paper=a4paper,left=33mm,right=33mm,top=38mm,bottom=38mm]{geometry}

\theoremstyle{plain}
\newtheorem{theorem}{Theorem}[section]
\newtheorem*{theorem*}{Theorem}

\newtheorem{lemma}[theorem]{Lemma}

\newtheorem{corollary}[theorem]{Corollary}

\newtheorem{proposition}[theorem]{Proposition}

\theoremstyle{definition}

\newtheorem*{remark}{Remark}

\newtheoremstyle{named}{}{}{\itshape}{}{\bfseries}{.}{.5em}{\thmnote{#3 }#1}
\theoremstyle{named}
\newtheorem*{namedtheorem}{Theorem}

\newcommand{\im}{\text{im}}

\newcommand{\Mod}{\text{Mod}}

\newcommand{\Aut}{\text{Aut}}

\newcommand{\Out}{\text{Out}}

\usepackage[backend=biber]{biblatex}
\usepackage{amsaddr}
\title{Prym representations of the Handlebody Group}
\author[Philipp Bader]{Philipp Bader}
\address{University of Glasgow, UK}
\email{p.bader.1@research.gla.ac.uk} 


\begin{document}

\begin{abstract}
Let $S$ be an oriented, closed surface of genus $g.$ The mapping class group of $S$ is the group of orientation preserving homeomorphisms of $S$ modulo isotopy. In 1997, Looijenga introduced the Prym representations, which are virtual representations of the mapping class group that depend on a finite, abelian group.\\[1ex]
Let $V$ be a genus $g$ handlebody with boundary $S$. The handlebody group is the subgroup of those mapping classes of $S$ that extend over $V.$ The twist group is the subgroup of the handlebody group generated by twists about meridians.\\[1ex]
Here, we restrict the Prym representations to the handlebody group and further to the twist group. We determine the image of the representations in the cyclic case.
\end{abstract}

\maketitle

\tableofcontents

\section*{Introduction}

The \textit{mapping class group} $\Mod(S)$ of a closed, oriented surface $S$ is the group of orientation preserving homeomorphisms of $S$ up to isotopy. Looijenga introduced the \textit{Prym representations} of the mapping class group (\cite{Loo}), which arise by acting on the homology of a finite, abelian covering of $S$. These (virtual) representations are only defined on a suitable finite index subgroup of $\Mod(S)$ consisting of certain liftable mapping classes. Furthermore, their image lies in a matrix group quotiented by a finite subgroup; see Section \ref{1.2} for a formal definition and more details. Throughout, we will use the term representation even when we map to a matrix group quotiented by a finite subgroup. Looijenga proceeds to determine the image of the Prym representations in the case of a cyclic covering, which then allows him to determine the image in the general abelian case up to finite index.\\[1ex]
If $S$ is the boundary of a handlebody $V,$ we call the subgroup of those homeomorphisms that extend to homeomorphisms of $V$ the \textit{handlebody group} and denote it by $\mathcal{H}_V(S)$. A different choice of handlebody $V'$ results in a conjugate subgroup $\mathcal{H}_{V'}(S).$ Throughout the paper, we fix a handlebody $V$ and consider $\mathcal{H}_V(S).$ The handlebody group arises naturally when studying Heegaard splittings: Given two handlebodies and a gluing along their boundary, the homeomorphism type of the resulting manifold is invariant under changing the gluing by composition with an element in the handlebody group. The purpose of this paper is to study certain representations of the handlebody group.\\[1ex]
Using an approach similar to Looijenga's, Grunewald, Larsen, Lubotzky and Malestein define representations for any finite, regular covering (\cite{LM}). The difference in their setup is that they use the \textit{punctured mapping class group} $\Mod(S, x_0),$ in order to have unique lifts for elements. The punctured mapping class group consists of homeomorphisms that fix $x_0$ up to isotopy that fixes $x_0$. The authors of \cite{LM} determine the image of this group as well as the image of the \textit{punctured handlebody group} $\mathcal{H}_V(S,x_0)$ (defined analogously) under their representations up to finite index.\\[1ex]
Our main result is to determine the precise image of (a finite index subgroup of) the handlebody group $\mathcal{H}_V(S)$ in the case of a cyclic covering. In this case, the representations of Looijenga, as well as of Grunewald, Larsen, Lubotzky and Malestein have the same image (up to quotienting by a finite cyclic group), so our result yields (virtual) representations of $\mathcal{H}_V(S)$, as well as $\mathcal{H}_V(S,x_0)$, where the precise image is known.
In particular, we show:
\begin{namedtheorem}[Main]
For every $d \in \mathbb{N},$ and every genus $g \ge 2,$ there is a finite index subgroup $\Gamma$ of the punctured handlebody group $\mathcal{H}_V(S,x_0)$ and a representation 
$$\Gamma \to GL_{2g-2}(\mathbb{Z}[\zeta_d]),$$ whose image is the subgroup
$$\Lambda := \biggl\{\begin{pmatrix}
    (D^*)^{-1} & B\\
    0 & D\\
\end{pmatrix} \, \biggl\mid \, \det(D) = \pm \zeta_d^k, D^*B = B^*D\biggl\}.$$ 
Here, $\zeta_d$ is a $d^{\text{th}}$ root of unity and $D^*$ is the adjoint matrix.\\[1ex]
Equivalently, there is a finite index subgroup of the (non-punctured) handlebody group $\mathcal{H}_V(S)$ that surjects onto $\Lambda/\langle \zeta_d \rangle,$ where $\langle \zeta_d \rangle$ is the subgroup generated by the diagonal matrix $$\begin{pmatrix}
\zeta_d & \cdots & 0  \\
\vdots & \ddots & \vdots\\
0 & \cdots & \zeta_d
\end{pmatrix}.$$
\end{namedtheorem}
We give two proofs of this theorem. Both proofs rely on a result of Grunewald and Lubotzky for an analogous graph-theoretic representation of the automorphism group of a free group (stated in \cite{GL}) as well as a result of Grunewald, Lubotzky, Larsen and Malestein relating the above representation to one of the handlebody group (stated in \cite{LM}). Our first proof comes from carefully examining Looijenga's computation for the image of the whole mapping class group and adjusting it to the handlebody group. In particular, we observe which mapping classes used in Looijenga's computation are in fact in the handlebody subgroup and prove that they are enough to deduce our result. For the second one, we determine the image of a subgroup of the handlebody group (called the twist group) and observe that using the results from \cite{GL} and \cite{LM}, it is enough to figure out this image. We show that there is a surjective representation of the twist group onto the subgroup
$$\biggl\{\begin{pmatrix}
    Id & B\\
    0 & Id\\
\end{pmatrix} \biggl\mid B = B^*\biggl\}$$ of $GL_{2g-2}(\mathbb{Z}[\zeta_d]).$ In particular, it follows that the twist group surjects onto the additive group of self-adjoint matrices with entries in $\mathbb{Z}[\zeta_d].$\\[1ex]

\textbf{Outline.}
In the first section, we recall some basics about the mapping class group, as well as the handlebody- and the twist group. Then, we define the representations that were studied in \cite{Loo}, respectively \cite{LM}, and explain how these representations are related. In the second section, we state our main results. In the third section, we prove our main theorem, which determines the image of the handlebody group. As a corollary, we obtain arithmetic quotients of (a finite index subgroup of) the handlebody group. In the genus $2$ case, we also obtain a virtual surjection of the handlebody group onto the integers. In the fourth section, we determine the image of the twist group which gives an alternative proof of our main theorem. Again, as a corollary, we obtain arithmetic quotients of the twist group and are in particular able to show that it surjects onto the integers.\\[1ex]
\textbf{Acknowledgements.}
A special thanks to Sebastian Hensel under whose helpful supervision during my Master's degree most part of this work was undertaken. I am very grateful to Vaibhav Gadre and Tara Brendle for their support and encouragement as well as many useful comments. I would like to thank Giulia Carfora, Riccardo Giannini and Isacco Nonino for valuable conversations and helpful advice for the creation of the figures used in this work. Finally, I particularly thank the anonymous referee for detailed comments that substantially improved the exposition.

\section{Preliminaries}\label{1}

We start in Section \ref{1.1} by defining the mapping class group, handlebody group and twist group and describe the standard symplectic representation of these groups. In Section \ref{1.2}, we proceed with discussing the representations studied by Looijenga as well as Grunewald, Larsen, Lubotzky and Malestein. These can be seen as a generalisation of the standard symplectic one.

\subsection{Mapping class group, handlebody group and twist group}\label{1.1}

A genus $g$ handlebody is a closed 3-ball with $g$ 1-handles attached. We fix a handlebody and denote it throughout by $V.$ The boundary of $V$ is a closed genus $g$ surface, which we denote by $S.$ The mapping class group $\Mod(S)$, resp. the handlebody group $\mathcal{H}_V(S),$ is the group of orientation preserving homeomorphisms of $S,$ resp. $V,$ up to isotopy. The handlebody group can be thought of as the homeomorphisms of the surface that extend to the handlebody and is naturally a subgroup of $\Mod(S).$ See (\cite{Hensel}, Section 3) for more details. We will refer to elements of both groups as mapping classes.\\[1ex]
Let $\alpha$ be a simple closed curve on $S.$ We say that $\alpha$ is a meridian, if it bounds a disk in the handlebody $V.$ An element $f \in \mathcal{H}_V(S)$ preserves the topological properties of the handlebody and therefore maps meridians to meridians. In fact this is a sufficient condition, in the sense that an $f \in \Mod(S)$ is in the handlebody group, if and only if $f$ maps meridians to meridians (see \cite{Hensel}, Corollary 5.11). Consequently it follows that a Dehn twist about a meridian is in $\mathcal{H}_V(S),$ whereas a Dehn twist about a simple closed curve that doesn't bound a disk in $V$ is not.\\[1ex]
Let $\mathcal{T}_V(S)$ denote the group generated by all twists about meridians. We call this group the twist group. It is a subgroup of the handlebody group, so we have the inclusions $\mathcal{T}_V(S) \subset \mathcal{H}_V(S) \subset \Mod(S),$ which are both of infinite index (see \cite{Hensel}, Corollary 5.4 for the handlebody group). For the twist group this follows from the following alternative description: The twist group can also be defined as the kernel of the surjective map $\mathcal{H}_V(S) \to \Out(\mathbb{F}_g)$ which arises by assigning a mapping class to its induced outer automorphism of the fundamental group $\pi_1(V) \cong \mathbb{F}_g$ (see \cite{Hensel}, Theorem 6.4).\\[1ex]
We now recall the standard symplectic representation. This representation is defined by letting $\Mod(S)$ act on the first homology of the surface $H_1(S).$ The first homology is isomorphic to $\mathbb{Z}^{2g}$ generated by the curves $E_{\pm 1}, ... , E_{\pm g}$ as in figure \ref{figure1}.

\begin{figure}[h]
\centering
\begin{tikzpicture}
%\draw[step=1cm,color=gray] (0,0) grid (14,6);%Uncomment this to get some helpful grid lines
\node[anchor=south west,inner sep=0] at (0,0){\includegraphics[scale=0.7]{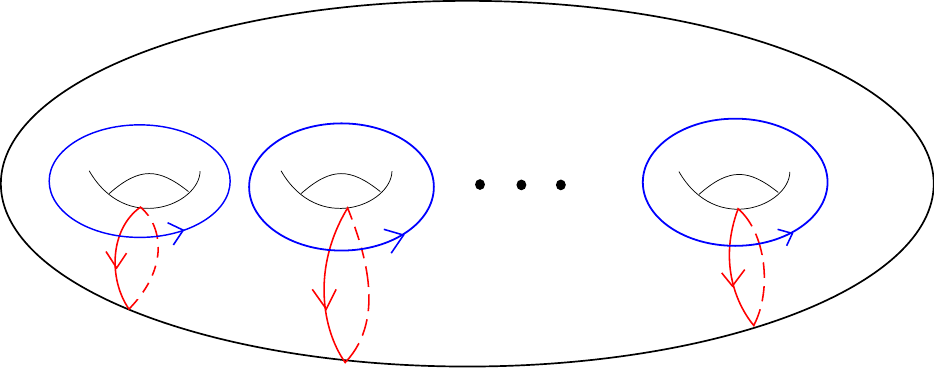}};
\node at (1.3,0.4) {$\textcolor{red}{E_1}$};
\node at (2.5,1.3) {$\textcolor{blue}{E_{-1}}$};
\node at (3.5,0.4) {$\textcolor{red}{E_2}$};
\node at (5.1,1.3) {$\textcolor{blue}{E_{-2}}$};
\node at (8.3,0.7) {$\textcolor{red}{E_g}$};
\node at (9.5,1.2) {$\textcolor{blue}{E_{-g}}$};
\end{tikzpicture}
\caption{The curves $E_{\pm i}$}
\label{figure1}
\end{figure}
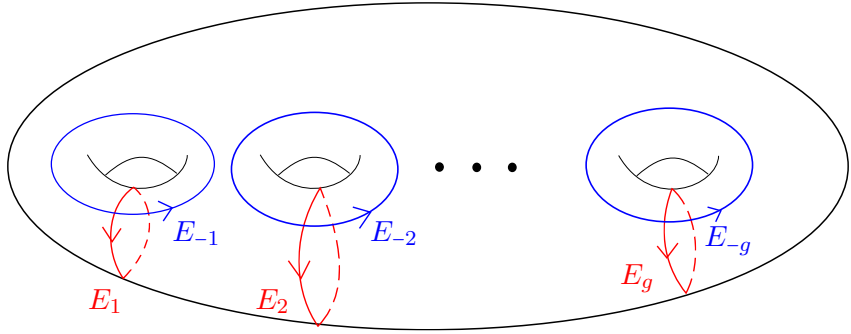

We will always use the convention that $V$ is the handlebody, so that the curves $E_1, ... , E_g$ are meridians, whereas $E_{-1}, ... , E_{-g}$ are not. We also choose the orientation of our surface to be so that the algebraic intersection is given by $(E_i, E_{-i}) = 1$ (and consequently $(E_{-i}, E_i) = -1$) for $i = 1, ... , g.$\\[1ex]
We state the following well known fact as a lemma, since it will be used later on.

\begin{lemma}\label{meridians}
If $\alpha$ is a meridian, then the homology class of $\alpha$ is contained in the subspace $\langle E_1, ... , E_g \rangle.$
\end{lemma}
\begin{proof}
We can write the homology class of $\alpha$ as a linear combination 
$$n_1E_1 + ... + n_gE_g + n_{-1}E_{-1} + ... + n_{-g}E_{-g},$$ 
where $n_{\pm i} \in \mathbb{Z}$ for all $i = 1, ... , g.$ We have to show that $n_{-1} = n_{-2} = ... = n_{-g} = 0.$ Assume that for some $i,$ we have $n_{-i} \neq 0.$ Then the algebraic intersection number between $\alpha$ and $E_i$ is non-zero contradicting Lemma 2.1 in \cite{Hensel}, which states that the albegraic intersection number between two meridians is always zero.
\end{proof}

We identify $\Aut(H_1(S))$ with $GL_{2g}(\mathbb{Z})$ via the ordered basis $E_1, ... , E_g, E_{-1}, ... , E_{-g}.$ The standard symplectic representation is thus a homomorphism 
$$\Mod(S) \to GL_{2g}(\mathbb{Z}).$$ 
Its image is the symplectic group $Sp_{2g}(\mathbb{Z}),$ which comes from the fact that the action of the mapping class group preserves the algebraic intersection form on $H_1(S)$ (see \cite{Primer}, Theorem 6.4).\\[1ex]
By restricting this representation to the handlebody group, we obtain a homomorphism 
$$\mathcal{H}_V(S) \to Sp_{2g}(\mathbb{Z}).$$ 
Since any $f \in \mathcal{H}_V(S)$ maps meridians to meridians, its action on $H_1(S)$ has to preserve the subspace generated by $E_1, ... , E_g$ (compare Lemma \ref{meridians}). So the image of the handlebody group under this representation consists of matrices of the form $$\begin{pmatrix}
    A & B\\
    0 & D\\
\end{pmatrix}.$$
Since these matrices are in $Sp_{2g}(\mathbb{Z}),$ one sees that in fact they have to be of the form $$\begin{pmatrix}
    (D^t)^{-1} & B\\
    0 & D\\
\end{pmatrix}$$ where $B$ satisfies $D^tB = B^tD.$\\[1ex]
Further restricting to the twist group gives us the representation 
$$\mathcal{T}_V(S) \to Sp_{2g}(\mathbb{Z}).$$ 
Since the algebraic intersection of any two meridians is $0$ (\cite{Hensel}, Lemma 2.1), any Dehn twist about a meridian will not change the homology class of any other meridian. Consequently, a generator of the twist group will fix the subspace generated by $E_1, ... , E_g$ pointwise. Therefore, a matrix in the image of the representation of the twist group is of the form $$\begin{pmatrix}
    Id & B\\
    0 & Id\\
\end{pmatrix}$$ with $B^t = B.$\\[1ex]
In fact, the image of the handlebody group as well as the image of the twist group under the standard symplectic representation consists of all matrices of the above discussed form, respectively. Our results generalise this to representations obtained by acting on the homology of a cyclic covering, which we define now.

\subsection{Definition of Prym representations}\label{1.2}
We start by constructing the representations defined in \cite{Loo}. These are defined for finite abelian groups, but we restrict to the case of cyclic groups in our definition.\\[1ex]
Let $S$ be a closed genus $g$ surface. From now on we assume throughout that $g \ge 2.$ Let $\widetilde{S} \to S$ be a normal covering with deck group $C \cong \mathbb{Z}/d\mathbb{Z}.$ This covering gives rise to the exact sequence
$$H_1(\widetilde{S}) \to H_1(S) \to C \to 0.$$

\begin{lemma}
Any surjection $H_1(S) \to C$ arises by taking the algebraic intersection with a primitive element in $H_1(S)$ and then reducing modulo $d.$ 
\end{lemma}
\begin{proof}
Let $\phi: H_1(S) \to C$ be surjective. Let $a_{\pm i} \in \mathbb{Z}$ be any preimages of $\phi(E_{\pm i})$ under the modulo $d$ map. Defining $\bar{\phi}: H_1(S) \to \mathbb{Z}$ by $\bar{\phi}(E_{\pm i}) := a_{\pm i}$ shows that $\phi$ factors through $\mathbb{Z}.$\\[1ex]
Consider now the element
$$e := a_{-1}E_1 - a_1E_{-1} \pm ... + a_{-g}E_g - a_gE_{-g}.$$
Then we have $(e, E_{\pm i}) = a_{\pm i} = \bar{\phi}(E_{\pm i}),$ which shows that $\phi$ is the composition of algebraic intersection with $e$ followed by reduction modulo $d.$\\[1ex]
Finally, since $\phi$ is surjective, we can choose the $a_{\pm i}$ so that $\bar{\phi}$ is surjective. Consequently $e$ has to be primitive, since otherwise $e = ke'$ for $e' \in H_1(S)$ and $k >1$ which would imply $\text{image}(\bar{\phi}) \subset k\mathbb{Z}.$
\end{proof}

The above lemma tells us that the surjection $H_1(S) \to C$ arising from the covering $\widetilde{S} \to S$ is given by algebraic intersection with a primitive element in $H_1(S)$ followed by reduction modulo $d.$ Primitive elements in $H_1(S)$ are represented by non-separating simple closed curves (see \cite{Primer}, Proposition 6.2), so we have to consider the algebraic intersecion with the homology class of such a curve. For our purposes, we can assume without loss of generality that this curve is $E_{g}.$ This is because for a choice of a different curve $\gamma =: \gamma_g,$ we can find curves $\gamma_{\pm 1}, \gamma_{\pm 2}, ... , \gamma_{\pm g}$ that form a symplectic basis of $H_1(S)$ and a homeomorphism that maps the homology classes of $E_{\pm i}$ to the homology classes of $\gamma_{\pm i}$ for all $i.$ Here, a symplectic basis means that $(\gamma_i, \gamma_{-i}) = 1$ for $i > 0$ and $(\gamma_i, \gamma_j) = 0$ for $|i| \neq |j|.$ The existence of the aforementioned homeomorphism follows from the surjectivity of the symplectic representation $\Mod(S) \to Sp_{2g-2}(\mathbb{Z}).$

\begin{figure}[h]
\centering
\begin{tikzpicture}
%\draw[step=1cm,color=gray] (0,0) grid (14,8);%Uncomment this to get some helpful grid lines
\node[anchor=south west,inner sep=0] at (0,0){\includegraphics[scale=0.8]{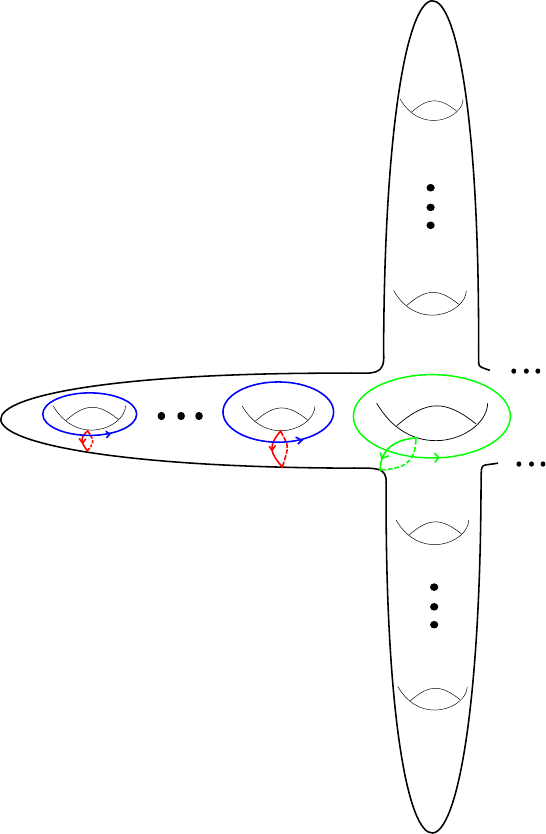}};
\node at (1.2,4.9) {$\textcolor{red}{e_1}$};
\node at (2,5.3) {$\textcolor{blue}{e_{-1}}$};
\node at (3.9,4.73) {$\textcolor{red}{e_{g-1}}$};
\node at (4.62,5.12) {$\textcolor{blue}{e_{-(g-1)}}$};
\node at (4.95,4.7) {$\textcolor{green}{e_g}$};
\node at (6,4.8) {$\textcolor{green}{e_{-g}}$};
\end{tikzpicture}
\caption{The covering space $\widetilde{S}$ and the curves $e_{\pm i}$}
\label{figure2}
\end{figure}
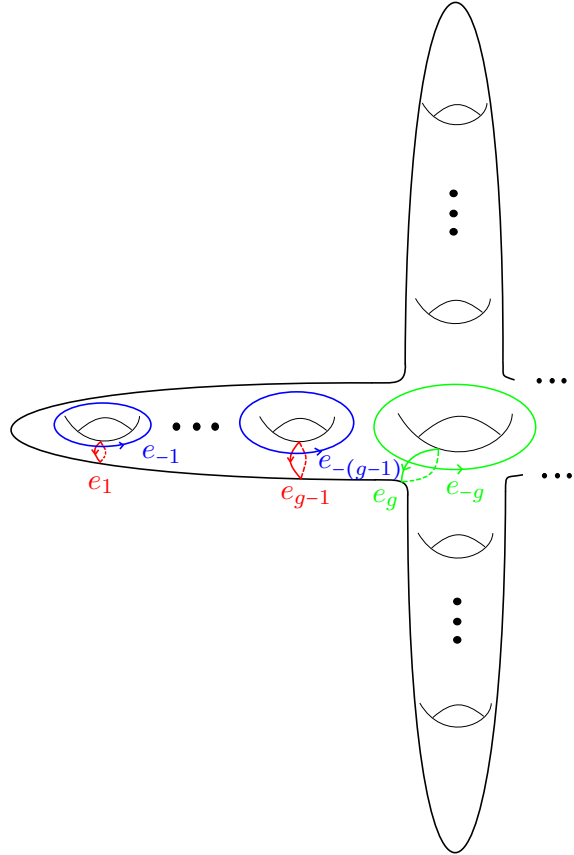

Hence, our covering corresponds to the surjection that maps all $E_{\pm i}$ to $0,$ except for $E_{-g}$ which is mapped to $1 \in \mathbb{Z}/d\mathbb{Z}.$ Geometrically, we can describe this covering as follows: Cut $S$ at $E_g$ and call the obtained genus $g-1$ surface with two boundary components $S'.$ Take $d$ copies of $S'$ and glue them in a cyclic order. The resulting surface can be seen in figure \ref{figure2}. The deck group $C$ acts as rotations and the curves $e_{\pm i}$ (seen as subsets of the surface) map to the $E_{\pm i}$ respectively under the covering projection. If we regard the $e_{\pm i}$ as homology classes, the previous sentence is true in all cases except for $e_{-g},$ which maps to $dE_{-g}.$\\[1ex]
Let $K = \ker(H_1(S) \to C).$ By acting on homology, every $f \in \Mod(S)$ yields an automorphism $f_*: H_1(S) \to H_1(S).$ If $f_*$ preserves $K,$ this automorphism induces an automorphism $\bar{f}$ of $H_1(S)/K \cong C.$\\[1ex]
Let $\Gamma_{S,C} \subset \Mod(S)$ denote the subgroup of mapping classes $f$ such that $f_*$ preserves $K$ and $\bar{f} = id.$ This is a finite index subgroup of $\Mod(S).$ We also want to use an equivalent characterisation of $\Gamma_{S,C}:$ 

\begin{lemma}\label{description of GammaSC}
$\Gamma_{S,C}$ is the group of mapping classes that admit a lift to $\widetilde{S}$ and whose lifts commute with the deck transformations.
\end{lemma}
\begin{proof}
Let $f \in \Gamma_{S,C}.$ That there is a lift $\tilde{f} \in \Mod(\widetilde{S})$ of $f$ follows from the lifting lemma in algebraic topology using the fact that $f_*$ preserves $K.$\\[1ex]
Now let $c \in C$ be a deck transformation, $x \in S$ a point and $\tilde{x} \in \widetilde{S}$ a preimage. By abuse of notation let $f, \tilde{f}$ be representing homeomorphisms of the corresponding mapping classes. Let $\alpha$ be a loop in $S$ based at the point $x$ whose homology class maps to $c$ under $H_1(S) \to C.$ This means that the lift $\tilde{\alpha}$ of $\alpha$ to the point $\tilde{x}$ has as endpoint $c(\tilde{x}).$ Since $\bar{f} = id,$ the homology class of $f(\alpha)$ also maps to $c$ under $H_1(S) \to C.$ Therefore the lift of $f(\alpha)$ to the point $\tilde{f}(\tilde{x})$ has as endpoint $c(\tilde{f}(\tilde{x})).$ But note that $\tilde{f}(\tilde{\alpha})$ is also a lift of $f(\alpha)$ which starts at $\tilde{f}(\tilde{x})$ and ends at $\tilde{f}(c(\tilde{x})).$ By the uniqueness of the lifting lemma, we conclude that the endpoints have to be the same, i.e. $c(\tilde{f}(\tilde{x})) = \tilde{f}(c(\tilde{x})).$ Since $\tilde{x}$ was arbitrary, we obtain $c \circ \tilde{f} = \tilde{f} \circ c.$\\[1ex]
This shows that any element of $\Gamma_{S,C}$ satisfies the description given in the lemma. For the converse let $f \in \Mod(S)$ such that there is a lift $\tilde{f}$ which commutes with the deck transformations. Since the diagram

\begin{center}
\begin{tikzcd}\label{diagram}
H_1(\widetilde{S}) \arrow{r}{\tilde{f}_*} \arrow{d} & H_1(\widetilde{S}) \arrow{d}\\
H_1(S) \arrow{r}{f_*} & H_1(S)
\end{tikzcd}
\end{center}

commutes and the image of the vertical maps is $K,$ we get that $f_*$ preserves $K.$\\[1ex]
Now let $\alpha$ be a loop and $c \in C$ the image of the homology class of $\alpha$ under $H_1(S) \to C.$ The lift $\tilde{\alpha}$ of $\alpha$ to some suitable point $\tilde{x}$ has as endpoint $c(\tilde{x}).$ Consider now the loop $f(\alpha).$ It's lift to the point $\tilde{f}(\tilde{x})$ is $\tilde{f}(\tilde{\alpha}),$ whose endpoint is $\tilde{f}(c(\tilde{x})),$ which by assumpion is equal to $c(\tilde{f}(\tilde{x})).$ This implies that the homology class of $f(\alpha)$ also maps to $c$ under $H_1(S) \to C.$ Since $\alpha$ was arbitrary, we conclude that $\bar{f} = id.$
\end{proof}

Let $\Gamma_{S,C}^\#$ be the group of lifts of elements in $\Gamma_{S,C}.$ Since any two lifts of the same element differ by a deck transformation, we obtain the short exact sequence stated in the following lemma.

\begin{lemma}
The sequence
$$1 \to C \to \Gamma_{S,C}^\# \to \Gamma_{S,C} \to 1$$
is short exact.
\end{lemma}
\begin{proof}
This statement can easily be checked on the level of homeomorphisms. On the level of mapping classes, from (\cite{Margalit}, Proposition 3.1) it follows that 
$$1 \to C \to \Gamma_{S}^\# \to \Gamma_{S} \to 1,$$
where $\Gamma_S$ is the group of all liftable mapping classes and $\Gamma_{S}^\#$ its group of lifts, is short exact. Since $\Gamma_{S,C} \subset \Gamma_S$ is a subgroup and $\Gamma_{S,C}^\#$ is the preimage of $\Gamma_{S,C}$ under $\Gamma_{S}^\# \to \Gamma_{S},$ the only thing remaining to check is whether $\text{image}(C \to \Gamma_{S}^\#)$ is contained in $\Gamma_{S,C}^\#.$ This is the case, since $C$ is abelian and therefore any element in $C$ commutes with all other elements in $C$, i.e. with every deck transformation. 
\end{proof}

By letting $\Gamma_{S,C}^\#$ act on $H_1(\widetilde{S}),$ we obtain a representation
$$\Gamma_{S,C}^\# \to \Aut(H_1(\widetilde{S})),$$ which in turn induces a representation
$$\Gamma_{S,C} \to \Aut(H_1(\widetilde{S}))/C.$$
The deck group $C$ acts on $H_1(\widetilde{S})$ and from (\cite{Loo}, Proposition 4.2) we know that the following is an isomorphism of $\mathbb{Z}[C]$-modules: 
$$H_1(\widetilde{S}) \cong \mathbb{Z}[C]^{2g-2} \oplus \mathbb{Z}^2.$$
The curves $e_{\pm 1}, ... , e_{\pm (g-1)}$ (compare figure \ref{figure2}) all generate a copy of $\mathbb{Z}[C]$ in homology, whereas the curves $e_{\pm g}$ both generate a copy of $\mathbb{Z}.$\\[1ex]  
Since the elements of $\Gamma_{S,C}^\#$ commute with the deck transformations, the image of the above representation consists of $C$-equivariant automorphisms, or equivalently maps into $\Aut_{\mathbb{Z}[C]}(\mathbb{Z}[C]^{2g-2} \oplus \mathbb{Z}^2).$\\[1ex]
Consider now homology with rational coefficients. By tensoring with $\mathbb{Q}$ on both sides, the above isomorphism becomes
$$H_1(\widetilde{S}; \mathbb{Q}) \cong \mathbb{Q}[C]^{2g-2} \oplus \mathbb{Q}^2.$$ 
By letting $\Gamma_{S,C}^\#$ act on rational homology, we obtain the representation
$$\Gamma_{S,C}^\# \to \Aut_{\mathbb{Q}[C]}(\mathbb{Q}[C]^{2g-2} \oplus \mathbb{Q}^2).$$
The reason we did this is because we can split $\mathbb{Q}[C]$ into simple submodules. In particular, if $C_1, ... , C_k$ denote all the cyclic factor groups of $C,$ then for all $i= 1, ... , k,$ $\mathbb{Q}(\zeta_i)$ with $\zeta_i$ a root of unity of order $|C_i|$ is a simple $\mathbb{Q}[C]$-module and $$\mathbb{Q}[C] \cong \bigoplus_{i = 1}^k \mathbb{Q}(\zeta_i),$$
compare (\cite{Loo}, Section 1). Note that $\mathbb{Q}(\zeta_i) = \mathbb{Q}$ for the simple submodule corresponding to the case of $C_i$ being the trivial group. It follows that the isotypical components of the $C$-representation $\mathbb{Q}[C]^{2g-2} \oplus \mathbb{Q}^2$ are $\mathbb{Q}(\zeta_i)^{2g-2}$ for $\zeta_i \neq 1$ and $\mathbb{Q}^{2g}.$ The representation of $\Gamma_{S,C}^\#$ descends to a representation on each isotypical component. In particular, after choosing the component corresponding to $C_i = C,$ we obtain 
$$\Gamma_{S,C}^\# \to \Aut_{\mathbb{Q}[C]}(\mathbb{Q}(\zeta)^{2g - 2}) \cong GL_{2g-2}(\mathbb{Q}(\zeta)),$$
where $\zeta$ is a $d^{\text{th}}$ root of unity. One can think of the curves $e_{\pm 1}, ... , e_{\pm(g-1)}$ as a basis of $\mathbb{Q}(\zeta)^{2g-2}.$ In order to identify $\Aut_{\mathbb{Q}[C]}(\mathbb{Q}(\zeta)^{2g - 2})$ with $GL_{2g-2}(\mathbb{Q}(\zeta)),$ we use the ordered basis $e_1, ... , e_g, e_{-1}, ... , e_{-g}.$\\[1ex]
We used $\mathbb{Q}$-coefficients in order to split our representation into pieces. However, the action of $\Gamma_{S,C}^\#$ on $H_1(\widetilde{S}; \mathbb{Q})$ is obtain from the action on $H_1(\widetilde{S})$ by tensoring with $\mathbb{Q}.$ Hence, the representation $\Gamma_{S,C}^\# \to \Aut_{\mathbb{Q}[C]}(H_1(\widetilde{S}; \mathbb{Q}))$ factors through $\Aut_{\mathbb{Z}[C]}(H_1(\widetilde{S})).$ Consequently, the matrices in the image of 
$$\Gamma_{S,C}^\# \to GL_{2g-2}(\mathbb{Q}(\zeta))$$
are ones with entries in the image of $\mathbb{Z}[C]$ under the projection $\mathbb{Q}[C] \to \mathbb{Q}(\zeta).$ This image is equal to $\mathbb{Z}[\zeta],$ the ring of integers of $\mathbb{Q}(\zeta).$\\[1ex] 
Let from now on $R = \mathbb{Z}[\zeta]$ with $\zeta$ a $d^{\text{th}}$ root of unity. The above discussion shows that we have a representation
\begin{equation}\label{rep1}
\Gamma_{S,C}^\# \to GL_{2g-2}(R).
\end{equation}
For more details about the whole construction, see \cite{Loo}.

\begin{remark}
Let $c \in C$ be a deck transformation. Then $c$ is a lift of the identity mapping class on $S,$ so $c \in \Gamma_{S,C}^\#.$ The action of $c$ on $\mathbb{Q}[C]^{2g-2} \oplus \mathbb{Q}^2$ is given by multiplication by $c$ on each $\mathbb{Q}[C]$ factor and the identity on $\mathbb{Q}.$ This can also be seen geometrically, since $c$ rotates the curves $e_{\pm 1}, ... , e_{\pm (g-1)}$ to the (for $c \neq 1$) non-homologous curves $c(e_{\pm 1}), ... , c(e_{\pm (g-1)}),$ while $c(e_{\pm g}) = e_{\pm g} \in H_1(\widetilde{S}).$ The projection $\mathbb{Q}[C] \to \mathbb{Q}(\zeta)$ maps $c$ to $\zeta^k$ for some $k.$ Hence, $c$ acts on $\mathbb{Q}(\zeta)^{2g-2}$ by multiplication by $\zeta^k.$ Equivalently, the image of $c$ under $\Gamma_{S,C}^\# \to GL_{2g-2}(R)$ is given by the matrix $\zeta^kId.$ We will use the convention that the deck transformation corresponding to a counterclockwise rotation in figure \ref{figure2} acts as multiplication by $\zeta.$
\end{remark}

The image of representation (\ref{rep1}) was studied by Looijenga:\\[1ex]
Choose any embedding of the number field $\mathbb{Q}(\zeta)$ into $\mathbb{C}.$ Such an embedding sends $\zeta$ to a primitive $d^{\text{th}}$ root of unity in $\mathbb{C}.$ Therefore, restriction of complex conjugation to $\mathbb{Q}(\zeta)$ induces an automorphism (that maps $\zeta$ to $\zeta^{-1}$) which is independent of the embedding. We call this automorphism also complex conjugation and since it preserves the ring of integers $R,$ it further induces a well defined notion of complex conjugation on $R.$\\[1ex] 
For $A \in GL_{2g-2}(R),$ let $A^*$ denote the matrix obtained by complex conjugation of each entry and transposition of the matrix. Let $U_{2g-2}(R) \subset GL_{2g-2}(R)$ denote the subgroup of matrices $A$ such that $A^* \Omega A = \Omega,$ where $$\Omega = \begin{pmatrix} 0 & Id \\ -Id & 0 \end{pmatrix}.$$
These are exactly the matrices that preserve the form on $R^{2g-2},$ that is induced from the $\mathbb{Z}[G]$-valued algebraic intersection form on $H_1(\widetilde{S})$ (see \cite{Loo} for more details). The form is the skew-Hermitian, sesquilinear (linear in the first entry, conjugate-linear in the second entry) form given by $\langle e_{i}, e_{-i} \rangle = 1$ for $i = 1, ... , g-1.$ We will refer to this form as the intersection form on $R^{2g-2}.$ Let $U_{2g-2}^\#(R)$ denote the subgroup of $U_{2g-2}(R)$ of matrices that have an even power of $\zeta$ as determinant. In (\cite{Loo},Theorem 2.4), Looijenga shows the following:

\begin{theorem}
The image of the representation (\ref{rep1}) is the subgroup $U_{2g-2}^\#(R).$
\end{theorem}

We now give the construction of the representations defined in \cite{LM}. Again, we restrict ourselves to cyclic groups, even though in this case the representations were defined for any finite group (satisfying some conditions).\\[1ex] 
Let $x_0 \in S$ be a fixed basepoint. The punctured mapping class group $\Mod(S,x_0)$ is the group of orientation preserving homeomorphisms that fix $p$ up to isotopy fixing $p$ at all times. The punctured handlebody and twist group are defined analogously.\\[1ex]
Let $\widetilde{S} \to S$ be the covering with deck group $C \cong \mathbb{Z}/d\mathbb{Z}$ as above and let $\tilde{x}_0$ be any preimage of $x_0$. The covering gives rise to the exact sequence
$$1 \to \pi_1(\widetilde{S}, \tilde{x}_0) \to \pi_1(S, x_0) \to C \to 1.$$
Let now $K = \ker(\pi_1(S, x_0) \to C).$ Exactly as before (now with $\pi_1(S,x_0)$ instead of $H_1(S)$), for any $f \in \Mod(S,x_0),$ we obtain an automorphism $f_*$ of $\pi_1(S, x_0).$ If $f_*$ preserves $K,$ this automorphism induces an automorphism $\bar{f}$ of $C.$ The reason we worked with the exact sequence on homology before was that an element in $\Mod(S)$ only induces an automorphism of $H_1(S)$ and not one of $\pi_1(S,x_0).$ Elements of the punctured mapping class group however do induce automorphisms of $\pi_1(S,x_0).$\\[1ex]
Let $\Gamma_{S,C,x_0} \subset \Mod(S,x_0)$ denote the subgroup of mapping classes $f$ such that $f_*$ preserves $K$ and $\bar{f} = id.$ This is a finite index subgroup of $\Mod(S,x_0),$ and analogously to the proof of Lemma \ref{description of GammaSC} one shows that $\Gamma_{S,C,x_0}$ is the group of mapping classes that admit a lift to $\widetilde{S}$ and whose lifts commute with the deck transformations.\\[1ex]
The difference to the construction in \cite{Loo} is that now we have a preferred lift, namely the one that fixes $\tilde{x}_0.$ So, by lifting an element $f \in \Mod(S,x_0)$ to the unique element $\tilde{f} \in \Mod(\widetilde{S}, \tilde{x}_0)$ and looking at the action of $\tilde{f}$ on $H_1(\widetilde{S}),$ we obtain a representation
$$\Gamma_{S,C,x_0} \to \Aut(H_1(\widetilde{S})),$$ which by the same arguments as before induces a representation
\begin{equation}\label{rep2}
\Gamma_{S,C,x_0} \to GL_{2g-2}(R),
\end{equation} 
where $R = \mathbb{Z}[\zeta]$ and $\zeta$ is a $d^{\text{th}}$ root of unity.\\[1ex]
We now describe the relation between the groups $\Gamma_{S,C}$ and $\Gamma_{S,C,x_0}$ and proceed to show that the two representations (\ref{rep1}) and (\ref{rep2}) have the same image.

\begin{lemma}
There is a short exact sequence
$$1 \to \pi_1(S,x_0) \to \Gamma_{S,C,x_0} \to \Gamma_{S,C} \to 1,$$
where the elements of $\pi_1(S,x_0)$ can be thought of as point-pushing maps.
\end{lemma}
\begin{proof}
Recall the Birman exact sequence $1 \to \pi_1(S,x_0) \to \Mod(S,x_0) \to \Mod(S) \to 1$ (see \cite{Primer}, Section 4.2 for more details). In order to prove the lemma, it suffices to show that any point-pushing map is contained in $\Gamma_{S,C,x_0}.$ Let $\gamma \in \pi_1(S,x_0).$ If we think of $\gamma$ as a point pushing map, it acts on $\pi_1(S,x_0)$ by conjugation. In particular, it preserves the normal subgroup $\pi_1(\widetilde{S}, \tilde{x}_0)$ and induces the identity on $C$ since $C$ is abelian. Therefore, $\gamma \in \Gamma_{S,C,x_0}.$
\end{proof}

\begin{lemma}\label{rel}
The image of the representation $\Gamma_{S,C,x_0} \to GL_{2g-2}(R)$ is the same as the image of $\Gamma_{S,C}^\# \to GL_{2g-2}(R),$ i.e. it is the subgroup $U_{2g-2}^\#(R).$
\end{lemma}
\begin{proof}
Let $\rho_1$ denote the representation of $\Gamma_{S,C}^\#$ and $\rho_2$ the one of $\Gamma_{S,C,x_0}.$ By construction of the representations, we have $\im(\rho_1) = C \cdot \im(\rho_2),$ where $C$ is the induced action of the deck group on $R^{2g-2},$ i.e. generated by multiplication by $\zeta.$

\begin{figure}[h]
\centering
\begin{tikzpicture}
%\draw[step=1cm,color=gray] (0,0) grid (14,10);%Uncomment this to get some helpful grid lines
\node[anchor=south west,inner sep=0] at (0,0){\includegraphics[scale=0.563]{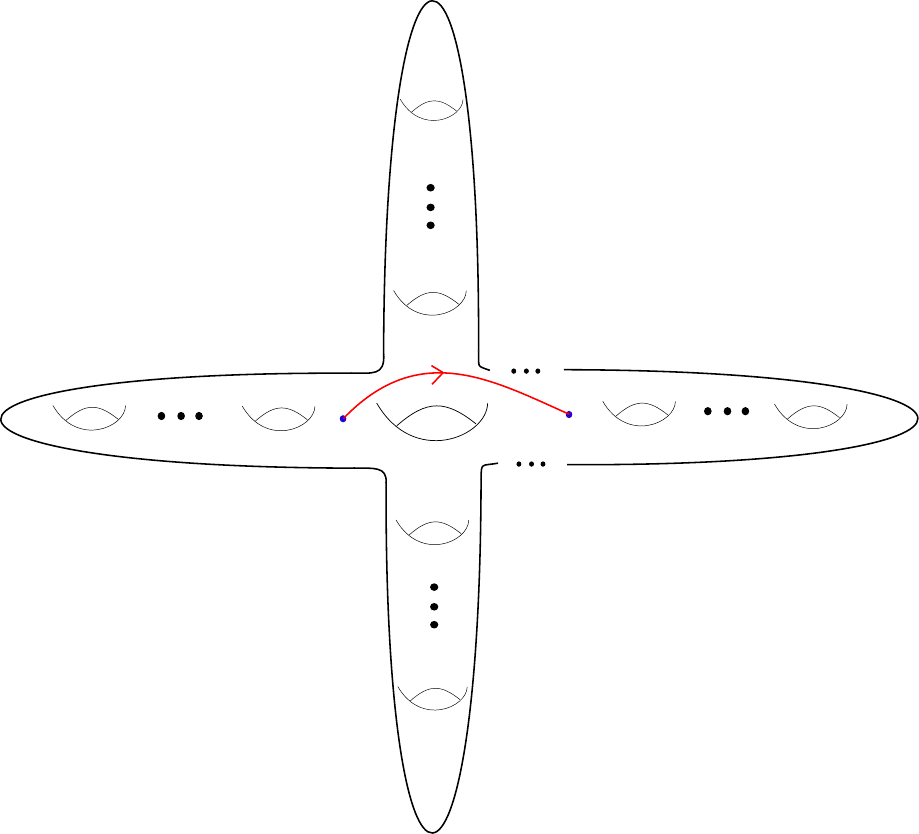}};
\node at (0.3,3.4) {$1$};
\node at (4.8,0.4) {$2$};
\node at (8.3,4.6) {$d-k+1$};
\node at (3.4,7.5) {$d$};
\node at (4.2,4.7) {$\textcolor{red}{\tilde{\alpha}}$};
\node at (3.3,3.7) {$\textcolor{blue}{\tilde{x}_0}$};
\node at (5.45,3.75) {$\textcolor{blue}{c^{-1} \cdot \tilde{x}_0}$};
\end{tikzpicture}
\caption{The lift of $\alpha$}
\label{figure3}
\end{figure}
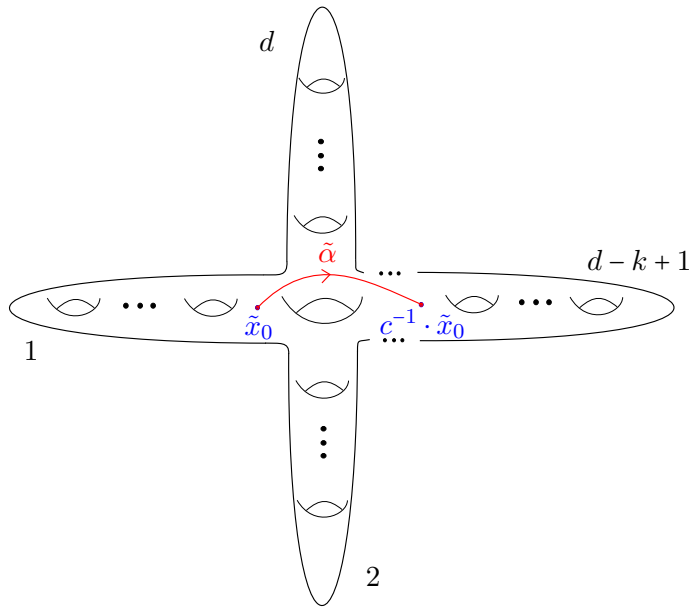

So, in order to prove the claim, we have to show that for any deck transformation $c \in C$ there is an element in $\Gamma_{S,C,x_0}$ such that the action of this element on $R^{2g-2}$ is the same as the one of $c.$ Let $c = k$ under the identification $C \cong \mathbb{Z}/d\mathbb{Z},$ i.e. $c$ is the deck transformation that rotates $\widetilde{S}$ so that each subsurface in figure \ref{figure2} moves $k$ steps further in a counterclockwise direction. So $c$ acts on $R^{2g-2}$ by multiplication by $\zeta^k.$ We show that there is $f \in \Gamma_{S,C,x_0}$ whose preferred lift to $\widetilde{S}$ acts also by multiplication by $\zeta^k.$\\[1ex]
Let $f$ be the point pushing map corresponding to the (at $x_0$ based) curve $\alpha := E_{-g}^{-k},$ i.e. the curve that winds $k$ times around $E_{-g}$ in the opposite direction. Then a lift $\tilde{\alpha}$ of $\alpha$ is the non-closed curve as in figure \ref{figure3}.\\[1ex]
A lift of $f$ is given by the point pushing map along $\tilde{\alpha}.$ However this lift maps $\tilde{x}_0$ to $c^{-1} \cdot \tilde{x}_0.$ So, in order to obtain a lift that fixes $\tilde{x}_0,$ we need to compose this point-pushing map with the deck transformation $c.$ This is now our preferred lift of $f$ and it clearly acts on $R^{2g-2}$ the same way that $c$ does.
\end{proof}

\section{Statements of the results}

We are now ready to formally formulate the results of this paper. This section contains the statements of the results as well as brief explanations on why these results are expected. All formal proofs will appear in later sections.\\[1ex]
Let $V$ be the handlebody with boundary surface $S$ as discussed in Section \ref{1.1} and $\mathcal{H}_V(S)$ the associated handlebody group. Let $\Gamma_{S,C}$ be as in Section \ref{1.2} and define $\Gamma_{V,C} := \Gamma_{S,C} \cap \mathcal{H}_V(S).$ Since $\Gamma_{S,C}$ is a finite index subgroup of $\Mod(S),$ it follows that $\Gamma_{V,C}$ is a finite index subgroup of the handlebody group. Let $\Gamma_{V,C}^\#$ denote the group of lifts of elements in $\Gamma_{V,C}.$ By restricting the representation (\ref{rep1}), we obtain the representation:
$$\Gamma_{V,C}^\# \to GL_{2g-2}(R),$$ 
which also induces the virtual represention of the handlebody group
$$\Gamma_{V,C} \to GL_{2g-2}(R)/C.$$
We repeat the main theorem stated in the introduction with more precise notation, which describes the image of these representations.

\begin{theorem}\label{main theorem}
The image of $\Gamma_{V,C}^\# \to GL_{2g-2}(R)$ is the subgroup 
$$\Lambda := \biggl\{\begin{pmatrix}
    (D^*)^{-1} & B\\
    0 & D\\
\end{pmatrix} \, \biggl\mid \, \det(D) = \pm \zeta^k, D^*B = B^*D\biggl\}.$$
\end{theorem}

The $0$-block in the lower left part of the matrices in $\Lambda$ comes from the fact that handlebody elements lift to handlebody elements (for an appropriate choice of handlebody for the cover) and these have to send meridians to meridians. The relation $D^*B = B^*D$ has to be satisfied, because the intersection form is preserved. That one obtains matrices with lower right block $D,$ where $\det(D) = \pm \zeta^k$ for some $k,$ follows from a result of Grunewald and Lubotzky (\cite{GL}). Hence, our main work consists in showing that for such a given $D,$ all matrices $$\begin{pmatrix}
    (D^*)^{-1} & B\\
    0 & D\\
\end{pmatrix}$$ with $D^*B = B^*D$ lie in the image of our representation, i.e. we get all possible upper right blocks $B.$\\[1ex]  
Theorem \ref{main theorem} generalizes the standard symplectic representation, in the sense that for $C = 0,$ we recover the result of the standard symplectic representation mentioned in the introduction. Note that the case $C=0$ corresponds to taking the identity as our covering and hence just acting on $H_1(S).$\\[1ex] 
Now, let $\mathcal{H}_V(S,x_0)$ denote the punctured handlebody group and define the subgroup $\Gamma_{V,C,x_0} := \Gamma_{S,C,x_0} \cap \mathcal{H}_V(S,x_0).$ Restricting the representation (\ref{rep2}) yields
$$\Gamma_{V,C,x_0} \to GL_{2g-2}(R).$$
Applying the same proof as in Lemma \ref{rel} to the representations of $\Gamma_{V,C}^\#$ and $\Gamma_{V,C,x_0}$ shows that their image is the same. Note that we are able to apply the same proof, since point pushing maps are in the handlebody group. So, Theorem \ref{main theorem} implies that the image of $\Gamma_{V,C,x_0} \to GL_{2g-2}(R)$ is also $\Lambda.$\\[1ex]
Let $\mathcal{T}_{V}(S)$ denote the twist group. Then $\mathcal{T}_{V}(S) \cap \Gamma_{S,C} = \mathcal{T}_{V}(S)$ (see Lemma \ref{twist lemma} below). Let $\mathcal{T}_{V}^{\#}(S)$ denote the group of lifts of elements in $\mathcal{T}_{V}(S).$ We thus obtain the representation 
$$\mathcal{T}_{V}^{\#}(S) \to GL_{2g-2}(R)$$ and prove:

\begin{theorem}
    The image of $\mathcal{T}^{\#}_{V}(S) \to GL_{2g-2}(R)$ is the subgroup $$\Delta := \biggl\{\zeta^k\begin{pmatrix}
    Id & B\\
    0 & Id\\
\end{pmatrix} \, \biggl\mid \, B = B^*\biggl\}.$$
\end{theorem}
That here the upper left and lower right blocks of the matrices in $\Delta$ are the identity comes from the fact that Dehn twists about meridians lift to a composition of Dehn twists about meridians and these leave the homology classes of every other meridian unchanged.

\section{Image of the Handlebody Group}

In this section, we prove our main theorem. We start by observing that the matrices in the image of our representation are of a certain block form, which allows us to split the discussion into two parts. In Section \ref{lrb}, we discuss the lower right block of the matrices and in Section \ref{urb}, the upper right block. We end with Section \ref{genus 2}, where we study the genus $2$ case more concretely.\\[1ex]
Let $V$ be the handlebody with boundary surface $S$ as defined in Section \ref{1.1}. Let $\widetilde{S} \to S$ be the cyclic covering defined in Section \ref{1.2} and observe that it was chosen in such a way that $\widetilde{S}$ bounds a handlebody $\widetilde{V}$ such that $\widetilde{V} \to V$ is also a cyclic covering with the same deck group $C = \mathbb{Z}/d\mathbb{Z}.$\\[1ex]
Let $\Gamma_{V,C}, \Gamma_{V,C,x_0}$ be the finite index subgroups of $\mathcal{H}_V(S), \mathcal{H}_{V,x_0}(S)$ respectively, defined in Section \ref{1.2}, let $\Gamma_{V,C}^\#$ be the group of lifts of $\Gamma_{V,C}$ and consider
$$\rho_1: \Gamma_{V,C}^\# \to GL_{2g-2}(R),$$
$$\rho_2: \Gamma_{V,C,x_0} \to GL_{2g-2}(R),$$
the associated representations also defined in Section \ref{1.2}, where $R = \mathbb{Z}[\zeta]$ and $\zeta$ is a $d^{\text{th}}$ root of unity.\\[1ex]
Since the image of the two representations $\rho_1, \rho_2$ is the same, for our purpose it doesn't matter which of the two we study. In the following we will use both of them interchangeably and whenever no confusion arises denote either of them simply by $\rho$.\\[1ex]
By construction of our covering, every $\Tilde{f} \in \Gamma_{V,C}^\#$ is contained in the handlebody group $\mathcal{H}_{\widetilde{V}}(\widetilde{S}).$ Therefore, $\Tilde{f}$ maps meridians to meridians, which means that its action on homology maps any $e_i$ ($i > 0$) into the subspace $\langle e_1, ... , e_{g-1} \rangle$ (compare lemma \ref{meridians}). Hence, the image of our representation only consists of upper right block matrices $$M = \begin{pmatrix}
    A & B\\
    0 & D
\end{pmatrix}.$$
Furthermore, the condition $M^* \Omega M = \Omega,$ where $$\Omega = \begin{pmatrix} 0 & Id \\ -Id & 0 \end{pmatrix},$$ implies that the matrices in the image are of the form $$M = \begin{pmatrix}
    (D^*)^{-1} & B\\
    0 & D
\end{pmatrix}$$ with $D^*B = B^*D.$

\subsection{The lower right block $D$}\label{lrb}
This chapter describes work from Grunewald and Lubotzky which can be found in (\cite{GL}). For completeness, we include it here, since our setup and notation are slightly different, and state their result.\\[1ex]
The lower right block $D$ of a matrix in the image of $\Gamma_{V,C,x_0}$ can be studied purely graph theoretically. Intuitively speaking, this comes from the fact that the handlebody deformation retracts onto a wedge of $g$ circles, whose homology classes are the $E_{-i}.$ Since $D$ only depends on the images of the $E_{-i},$ it turns out that studying an analogous graph theoretic representation is sufficient.\\[1ex]
Formally, let $X$ be a graph with one vertex $v$ and $g$ edges $x_1, ... , x_g$. Then $\pi_1(X,v) = \mathbb{F}_g,$ the free group on $g$ generators. Let $\widetilde{X}$ be the covering associated to the map $\mathbb{F}_g \to C$ which maps every generator $x_i$ to $0$ for $i \le g-1$ and maps $x_g$ to $1 \in C.$ The graph $X$ can be embedded into $V$ such that the edges $x_i$ are representatives of the homology classes of the $E_{-i}.$ From the definition of the covering, it follows that $\widetilde{X}$ is isomorphic as a covering space to the preimage of $X \subset V$ under $\widetilde{V} \to V.$ This preimage is a deformation retract of $\widetilde{V}.$ \\[1ex]
Let $K := \ker(\mathbb{F}_g \to C)$ and let $\Gamma_{X,C} \subset \Aut(\mathbb{F}_g)$ be the subgroup that preserves $K$ and induces the identity on $\mathbb{F}_g/K.$ Since $\Gamma_{X,C}$ preserves $K,$ it acts on $K.$ This action induces an action on $K/[K,K] \cong H_1(\widetilde{X}) \cong H_1(\widetilde{V}).$ 

\begin{lemma}
$H_1(\widetilde{V}) \cong \mathbb{Z}[C]^{g-1} \oplus \mathbb{Z}$ as $C$-representations (or equivalently $\mathbb{Z}[C]$-modules).
\end{lemma}
\begin{proof}
Recall that as a $\mathbb{Z}[C]$-module $H_1(\widetilde{S}) \cong \mathbb{Z}[C]^{2g-2} \oplus \mathbb{Z}^2,$ where the curves $e_{\pm 1}, ... , e_{\pm (g-1)}$ from figure \ref{figure2} form a basis of the free part, and the curves $e_{\pm g}$ generate the $\mathbb{Z}^2$ part. The inclusion $\widetilde{S} \hookrightarrow \widetilde{V}$ induces a surjection $H_1(\widetilde{S}) \to H_1(\widetilde{V}).$ Let $K$ be the kernel of this surjection. The $\mathbb{Z}$-rank of $H_1(\widetilde{V})$ and hence also the $\mathbb{Z}$-rank of $K$ is half of that of $H_1(\widetilde{S}).$ The curves $e_i$ for $i > 0$ lie in $K,$ since they are meridians. We obtain that the copy of $\mathbb{Z}[C]^{g-1} \oplus \mathbb{Z}$ generated by $e_1, ... , e_g$ is contained in the kernel and by comparing the ranks $K \cong \mathbb{Z}[C]^{g-1} \oplus \mathbb{Z}.$ Furthermore, $H_1(\widetilde{S})$ splits as the direct sum of $K$ and the submodule generated by $e_{-1}, ... , e_{-g},$ which is also isomorphic to $\mathbb{Z}[C]^{g-1} \oplus \mathbb{Z}.$ Hence, the short exact sequence
$$0 \to K \to H_1(\widetilde{S}) \to H_1(\widetilde{V}) \to 0$$
can be rewritten as the following short exact sequence of $\mathbb{Z}[C]$-modules:
$$0 \to K \to K \oplus \mathbb{Z}[C]^{g-1} \oplus \mathbb{Z} \to H_1(\widetilde{V}) \to 0,$$ 
where the inclusion of $K$ into $K \oplus \mathbb{Z}[C]^{g-1} \oplus \mathbb{Z}$ is the natural inclusion $k \mapsto (k, 0).$ From this, we get the desired result
$$H_1(\widetilde{V}) \cong  (K \oplus \mathbb{Z}[C]^{g-1} \oplus \mathbb{Z})/K \cong \mathbb{Z}[C]^{g-1} \oplus \mathbb{Z}.$$
\end{proof}

From the lemma, we obtain that we have an action of $\Gamma_{X,C}$ on $\mathbb{Z}[C]^{g-1} \oplus \mathbb{Z}$ and consequently on $R^{g-1}.$ In other words we obtain a representation $\eta: \Gamma_{X,C} \to GL_{g-1}(R).$\\[1ex]
The punctured handlebody group $\mathcal{H}_V(S,x_0)$ acts on $\pi_1(V,x_0) \cong \mathbb{F}_g$ and this action induces a surjective map $\mathcal{H}_V(S,x_0) \to \Aut(\mathbb{F}_g)$ (see \cite{Hensel}, Theorem 6.3). Since $\Gamma_{V,C,x_0}$ by definition acts on $\mathbb{F}_g$ by preserving $K$ and inducing the identity on $\mathbb{F}_g/K$, the image of $\Gamma_{V,C,x_0}$ under the above map is contained in $\Gamma_{X,C}.$ Given an automorphism $h \in \Gamma_{X,C},$ there is some $f \in \mathcal{H}_V(S,x_0)$ which acts on $\mathbb{F}_g$ as $h$ and consequently $f$ lies in $\Gamma_{V,C,x_0}.$ This implies that the image of $\Gamma_{V,C,x_0}$ under $\mathcal{H}_V(S,x_0) \to \Aut(\mathbb{F}_g)$ is all of $\Gamma_{X,C}.$\\[1ex]
The following is a fact established in (\cite{LM}, Lemma 6.6).

\begin{lemma}
Let $\rho: \Gamma_{V,C,x_0} \to GL_{2g-2}(R)$ and $\eta: \Gamma_{X,C} \to GL_{g-1}(R)$ denote our two representations. For any $f \in \Gamma_{V,C, x_0},$ let $f_*$ denote the induced automorphism of $\mathbb{F}_g.$ This yields a map $\Gamma_{V,C,x_0} \to \Gamma_{X,C}.$ Let $$pr: \begin{pmatrix}
    (D^*)^{-1} & B\\
    0 & D
\end{pmatrix} \mapsto D$$ denote the projection to the lower right block. Then it holds that
$$pr(\rho(f)) = \eta(f_*),$$
i.e. the following diagram commutes:
\begin{center}
\begin{tikzcd}\label{diagram}
\Gamma_{V,C,x_0} \arrow{r}{\rho} \arrow{d}{} & \biggl\{\begin{pmatrix}
    (D^*)^{-1} & B\\
    0 & D
\end{pmatrix} \, \biggl\mid \, D \in GL_{g-1}(R), D^*B = B^*D\biggl\} \arrow{d}{pr}\\
\Gamma_{X,C} \arrow{r}{\eta} & \{D \, | \, D \in GL_{g-1}(R)\}.
\end{tikzcd}
\end{center}
\end{lemma}

Because the vertical maps in the diagram of Lemma \ref{diagram} are surjective, a matrix $D$ occurs as a lower right block in the representation of $\Gamma_{V,C,x_0},$ if and only if it is in the image of $\eta.$ So, in order to find those matrices $D,$ it suffices to determine the image of the graph theoretic representation of $\Gamma_{X,C}.$ This is done in \cite{GL}, where Grunewald and Lubotzky study virtual representations of $\Aut(\mathbb{F}_g)$ by looking at the action on the homology of finite coverings of graphs. In particular, they show: 

\begin{theorem}
The image of $\Gamma_{X,C} \to GL_{g-1}(R)$ consists of all matrices $D$ with determinant $\pm \zeta^k$ for some $k \in \mathbb{Z}.$ 
\end{theorem}
\begin{proof}
Consider the homomorphism $\Aut(\mathbb{F}_g) \to GL_{g}(\mathbb{Z}) \to \{\pm 1\},$ where the first map comes from the abelianisation functor and the second is the determinant. Let $\Gamma_{X,C}^+$ be the subgroup of $\Gamma_{X,C}$ of elements that map to $+1$ under the above homomorphism. This is a subgroup of index $2.$ Let $f \in \Gamma_{X,C} \setminus \Gamma_{X,C}^+$ be the automorphism that maps $x_1 \mapsto x_1^{-1}$ and fixes every other generator $x_i \in \mathbb{F}_g.$\\[1ex] 
The image of $\Gamma_{X,C}^+$ under $\Gamma_{X,C} \to GL_{g-1}(R)$ consists of all matrices with determinant a power of $\zeta$ (see \cite{GL}, Proposition 6.4). The element $f$ maps under our representation to the diagonal matrix 
$$\begin{pmatrix}
-1 &  &  & \\
 & 1 &  & \\
 &  & \ddots & \\
 &  &  &1\\
\end{pmatrix},$$
which has determinant $-1.$ Hence, the image of $\Gamma_{X,C}$ consists of all matrices with determinant a power of $\zeta$ as well as with determinant $-1$ times a power of $\zeta.$
\end{proof}

We denote the subgroup of $GL_{g-1}(R)$ of matrices $D$ with $\det(D) = \pm \zeta^k$ by $GL_{g-1}^{\pm}(R).$ Note that when the order of $k$ is even, the minus sign is redundant. For us, the theorem implies:

\begin{corollary}
The matrices that occur as a lower right block in the representation of $\Gamma_{V,C,x_0}$ (or the one of $\Gamma_{V,C}^\#$) are exactly the matrices $D \in GL_{g-1}^{\pm}(R)$.
\end{corollary}

\subsection{The upper right block $B$}\label{urb}

Given some $D \in GL_{g-1}^{\pm}(R),$ we know that there is a matrix $$\begin{pmatrix}
    (D^*)^{-1} & B\\
    0 & D
\end{pmatrix}$$ in the image of the representation of $\Gamma_{V,C}^\#$, but we don't know which $B$ can occur.\\[1ex] 
A necessary condition for $B$ is to satisfy the equality $D^*B = B^*D.$ The goal of this section is to show that for any $B$ satisfying $D^*B = B^*D,$ the matrix $$\begin{pmatrix}
    (D^*)^{-1} & B\\
    0 & D
\end{pmatrix}$$ is in the image.\\[1ex]
Note that in order to show the goal, it is enough to consider the case where $D$ is the identity matrix $Id$:

\begin{lemma}
    If the subgroup $$\biggl\{\begin{pmatrix}
        Id & B\\
        0 & Id
    \end{pmatrix} \, \biggl\mid \, B = B^*\biggl\}$$ is in the image, then any $$\begin{pmatrix}
        (D^*)^{-1} & B\\
        0 & D
    \end{pmatrix}$$ with $\det(D) = \pm \zeta^k$ and $D^*B = B^*D$ is in the image.
\end{lemma}

\begin{proof}
    Let $D \in GL_{g-1}^{\pm}(R)$ and let $B$ be a $(g-1) \times (g-1)$ matrix such that $D^*B = B^*D.$\\[1ex] 
    Since $D \in GL_{g-1}^{\pm}(R),$ we know from the commutative diagram in Lemma \ref{diagram} that there is some matrix $E$ with $D^*E = E^*D$ such that $$\begin{pmatrix}
        (D^*)^{-1} & E\\
        0 & D
    \end{pmatrix}$$ is in the image.\\[1ex]
    Define $F := D^*(B-E)$. Then $F$ is self-adjoint, since 
$$F^* = (B-E)^*(D^*)^* = B^*D - E^*D = D^*B - D^*E = D^*(B-E) = F.$$ 
Hence $$\begin{pmatrix}
Id & F\\
0 & Id\\ 
\end{pmatrix}$$ lies in the image by assumption. We compute $$
\begin{pmatrix}
(D^*)^{-1} & E\\
0 & D\\ 
\end{pmatrix}\begin{pmatrix}
Id & F\\
0 & Id\\ 
\end{pmatrix}=\begin{pmatrix}
(D^*)^{-1} & (D^*)^{-1}F + E\\
0 & D\\ 
\end{pmatrix}=\begin{pmatrix}
(D^*)^{-1} & B\\
0 & D\\ 
\end{pmatrix}$$
and conclude that $$\begin{pmatrix}
(D^*)^{-1} & B\\
0 & D\\ 
\end{pmatrix}$$ is also in the image.
\end{proof}

Because of the lemma, the goal of this section is now to prove the following:

\begin{proposition}\label{prop}
The image of the representation $\Gamma_{V,C}^\# \to GL_{2g-2}(R)$ contains the subgroup $$\biggl\{\begin{pmatrix}
        Id & B\\
        0 & Id
    \end{pmatrix} \, \biggl\mid \, B=B^*\biggl\}.$$
\end{proposition}

The rest of this section is devoted to proving the proposition. Our proof largely follows Looijenga's proof in (\cite{Loo}), where we make the necessary adjustments, since we work with the handlebody group instead of the whole mapping class group.\\[1ex] 
As it will turn out in Section \ref{4}, the subgroup $$\biggl\{\begin{pmatrix}
        Id & B\\
        0 & Id
    \end{pmatrix} \, \biggl\mid \, B=B^*\biggl\}$$ is exactly the image of our representation restricted to the twist group, which in particular proves the proposition. However, in this section we present a different proof.\\[1ex]
Let $\mathcal{P} \subset Sp_{2g-2}(\mathbb{Z})$ denote the upper right block matrices in the symplectic group, i.e. matrices of the form $$\begin{pmatrix}
(D^t)^{-1} & B\\
0 & D
\end{pmatrix}$$ where each block is a $(g-1) \times (g-1)$ matrix. We denote our representation of $\Gamma_{V,C}^\#$ by $\rho.$ 
\begin{lemma}\label{symplectic in image}
We have the inclusion $\mathcal{P} \subset \text{image}(\rho).$ 
\end{lemma} 
\begin{proof}
Let $\gamma$ be the curve separating $S$ into a genus $g-1$ surface with one boundary component $S'$ and a one holed torus $T$, as in figure \ref{figure4}.

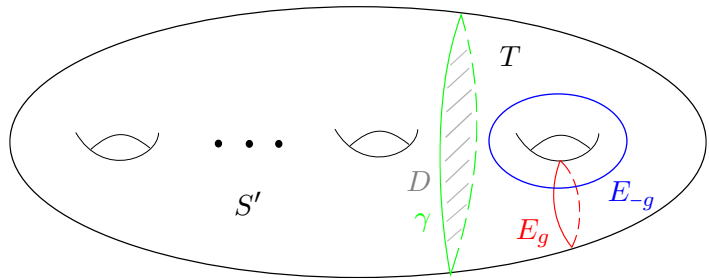
\begin{figure}[h]
\centering
\begin{tikzpicture}
%\draw[step=1cm,color=gray] (0,0) grid (9,4);%Uncomment this to get some helpful grid lines
\node[anchor=south west,inner sep=0] at (0,0){\includegraphics[scale=0.58]{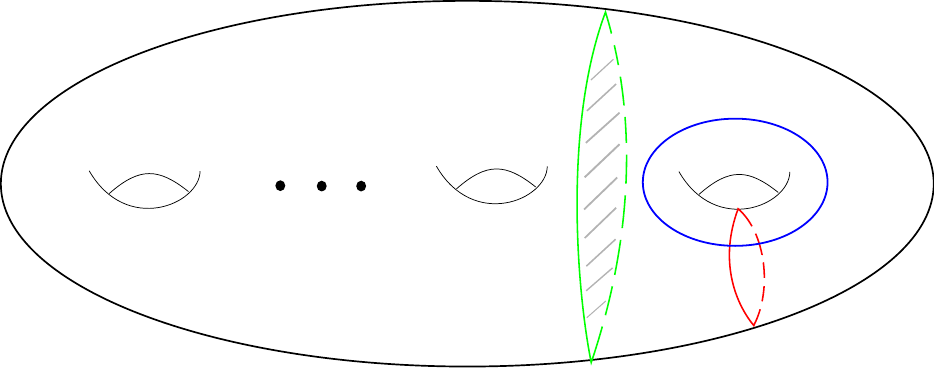}};
\node at (3.15,1) {$S'$};
\node at (6.6,2.95) {$T$};
\node at (6.9,0.6) {$\textcolor{red}{E_g}$};
\node at (8.2,1.1) {$\textcolor{blue}{E_{-g}}$};
\node at (5.45,0.75) {$\textcolor{green}{\gamma}$};
\node at (5.4,1.3) {$\textcolor{gray}{D}$};
\end{tikzpicture}
\caption{The subsurfaces $S'$ and $T$}
\label{figure4}
\end{figure}

Let $D$ be a disk embedded in the handlebody $V$ with boundary $\gamma$ as in figure \ref{figure4}. Let $S'' := S' \cup D.$ Then $S''$ is a closed surface that is embedded in $V$ and bounds a handlebody $V'' \subset V.$ The homology $H_1(S'')$ is generated by the $E_{\pm i}$ for $i = 1, ... , g-1$ and can therefore be identified with the subspace $\langle E_{\pm1}, ... , E_{\pm(g-1)} \rangle$ of $H_1(S).$ The image under the standard symplectic representation of the handlebody group $\mathcal{H}_{V''}(S'')$ is $\mathcal{P}$ by (\cite{Hensel}, Theorem 7.1). In order to prove the lemma, let $A \in \mathcal{P}$ and let $f'' \in \mathcal{H}_{V''}(S'')$ which maps to $A$ under the standard symplectic representation.\\[1ex]  
Any element of the handlebody group of $S''$ can be homotoped, such that it is the identity on $D$, and can hence be restricted to a homeomorphism of $S'$. By doing this for $f'',$ we obtain a homeomorphism of $S'$ which we call $f'.$ This in turn can be extended to a homeomorphism $f$ of $S$, where $f|_T$ is defined as the identity. Note that $f$ (seen as a mapping class) is in the handlebody group of $S$ and its action on the subspace $\langle E_{\pm1}, ... , E_{\pm(g-1)} \rangle \subset H_1(S)$ is the same as the action of $f''$ on $H_1(S'')$. Since $f$ restricted to $T$ is the identity, the induced action on the subspace $\langle E_g, E_{-g} \rangle \subset H_1(S)$ is also the identity.  Consequently $f$ (as a mapping class) lies in $\Gamma_{V,C}$.\\[1ex] 
Consider $\widetilde{S}$ as in figure \ref{figure5}.

\begin{figure}[h]
\centering
\begin{tikzpicture}
%\draw[step=1cm,color=gray] (0,0) grid (11,5);%Uncomment this to get some helpful grid lines
\node[anchor=south west,inner sep=0] at (0,0){\includegraphics[scale=0.7]{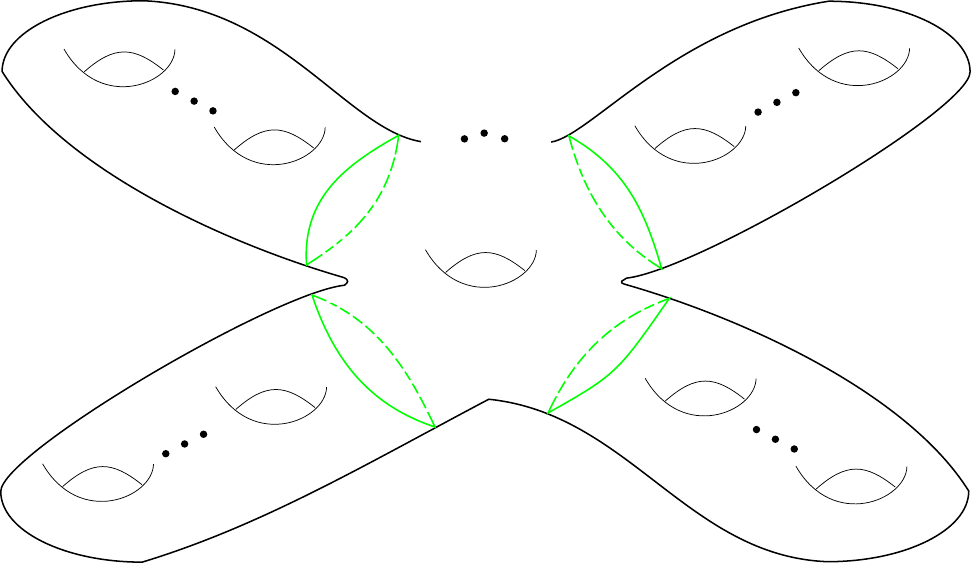}};
\node at (4.2,1.6) {$S_1$};
\node at (7.5,1.5) {$S_2$};
\node at (8,4.2) {$S_3$};
\node at (3.3,4.2) {$S_d$};
\node at (5.9,2.8) {$M$};
\end{tikzpicture}
\caption{The subsurfaces $S_i$ and $M$}
\label{figure5}
\end{figure}
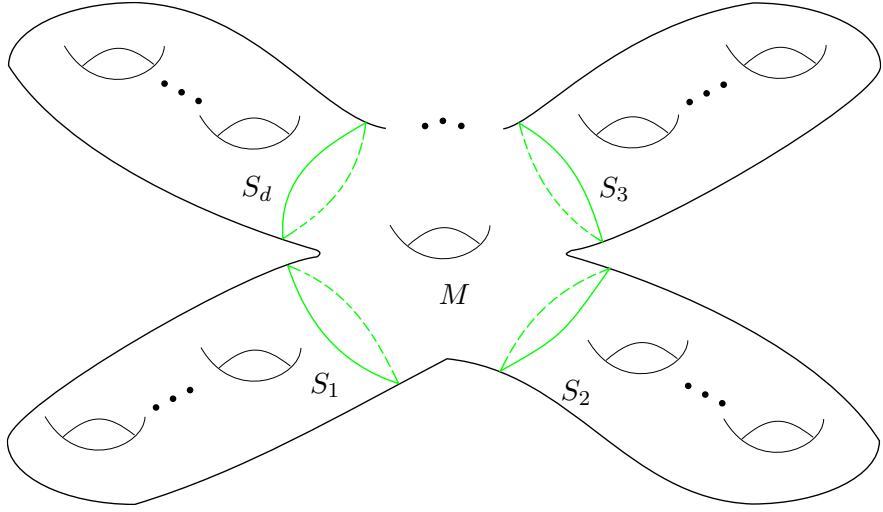

The green curves are the lifts of $\gamma$. They separate $\widetilde{S}$ into the surfaces $S_1, ... , S_d$ and the "middle surface" $M$ which is a torus with $d$ boundary components. All the $S_k$ for $k = 1, ... , d$ map homeomorphically onto $S'$ under the covering projection. Let $f_k: S_k \to S'$ be the homeomorphisms obtained by restricting the covering projection to the $S_k,$ which we use to identify each $S_k$ with $S'.$\\[1ex]
A lift $\tilde{f}$ of $f$ can be defined as $f_k^{-1} \circ f' \circ f_k$ on every $S_k$ and the identity on $M$. This is well defined since $f'$ is the identity on the boundary of $S'$ and $f_k$ maps the boundary of $S_k$ to the boundary of $S'$ for all $k = 1, ... , d.$\\[1ex]
The homology $H_1(S_k)$ is isomorphic through $f_k$ to $H_1(S')$ which in turn is naturally isomorphic to $H_1(S'') = \langle E_{\pm1}, ... , E_{\pm(g-1)} \rangle$ for all $k = 1, ... , d.$ Furthermore, letting $e_{\pm i}$ be the homology classes of the curves on $S_1$ as depicted in figure \ref{figure2} and $c^k$ be the deck transformation rotating $S_1$ to $S_k$ for all $k$, we have $H_1(S_k) = \langle c^ke_{\pm 1}, ... , c^ke_{\pm (g-1)} \rangle$ and the above mentioned isomorphism maps each $c^ke_{\pm i}$ to $E_{\pm i}$ for all $i = 1, ... , g-1.$ The action of $\tilde{f}$ on $H_1(S_k)$ is the same as the action of $f''$ on $H_1(S'')$ in the sense that the following diagram 

\begin{center}
\begin{tikzcd}
H_1(S_k) = \langle c^ke_{\pm 1}, ... , c^ke_{\pm (g-1)} \rangle \arrow{r}{\tilde{f}_*} \arrow{d} & \langle c^ke_{\pm 1}, ... , c^ke_{\pm (g-1)} \rangle = H_1(S_1) \arrow{d}\\
H_1(S'') = \langle E_{\pm1}, ... , E_{\pm(g-1)} \rangle \arrow{r}{f''_*} & \langle E_{\pm1}, ... , E_{\pm(g-1)} \rangle = H_1(S''),
\end{tikzcd}
\end{center}

where the vertical maps are the above mentioned isomorphism, commutes.\\[1ex]
In particular, the matrix corresponding to $\tilde{f}_*$ relative to the basis $c^ke_{\pm 1}, ... , c^ke_{\pm (g-1)}$ is the same as the one corresponding to $f''_*$ relative to the basis $E_{\pm 1}, ... , E_{\pm (g-1)},$ which is $A.$ Finally, since $\widetilde{f}$ preserves the $S_k,$ and the curves $e_{\pm 1}, ... , e_{\pm (g-1)}$ get mapped to $\mathbb{Z}$-linear combinations of themselves, the action of $\tilde{f}$ on $R^{2g-2}$ is also given by $A.$ This means $A$ lies in the image of $\tilde{f}$ in $GL_{2g-2}(R)$.\\[1ex]
Since using this procedure we can obtain any element of $\mathcal{P}$, this completes the proof.  
\end{proof}
Let $e_1, ... , e_{g-1}, e_{-1}, ... , e_{-(g-1)}$ be the standard basis of $R^{2g-2}$ and let $T: R^{2g-2} \to R^{2g-2}$ be defined by $e_{\pm1} \mapsto \zeta e_{\pm1}$ and $e_{\pm i} \mapsto e_{\pm i}$ for $i \ge 2$. Looijenga proves that $T$ is in the image of the representation of the mapping class group. Here, we show that $T$ is in fact also in the image of the representation of the handlebody group.
\begin{lemma}\label{T}
$T$ is in the image of $\rho$.
\end{lemma}
\begin{proof}
In order to prove that $T$ is in the image of $\Gamma_{S,C}^\#$, Looijenga defines the mapping class $\tau := T_{E_{-g}} \circ T_\alpha^{-1} \in \Gamma_{S,C}$ and shows that a lift of $\tau$ maps to $T$ (see \cite{Loo}, Proof of (2.4)).\\[1ex] 
Here, $T_\gamma$ denotes the Dehn twist about a curve $\gamma$ and $\alpha$ is the curve that can be seen in figure \ref{figure6}.\\[1ex]

\begin{figure}[h]
\centering
\begin{tikzpicture}
%\draw[step=1cm,color=gray] (0,0) grid (11,5);%Uncomment this to get some helpful grid lines
\node[anchor=south west,inner sep=0] at (0,0){\includegraphics[scale=0.7]{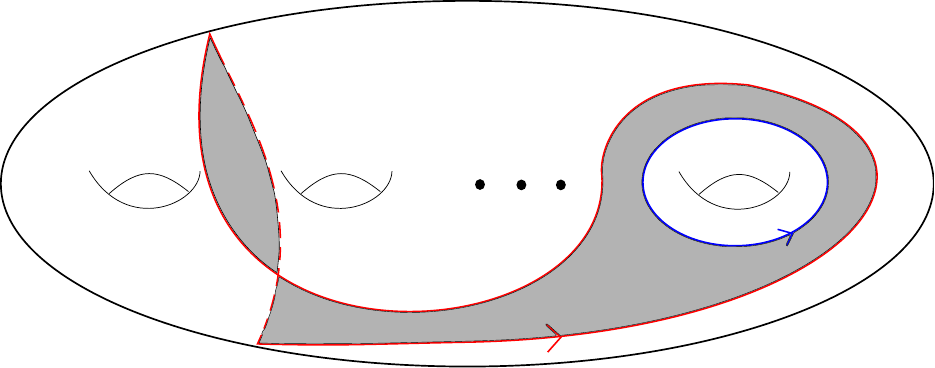}};
\node at (5.7,0.15) {$\textcolor{red}{\alpha}$};
\node at (8.75,1.66) {$\textcolor{blue}{E_{-g}}$};
\end{tikzpicture}
\caption{The curves $\alpha$ and $E_{-g}$ bound an annulus in the handlebody}
\label{figure6}
\end{figure}

Hence, in order to prove the lemma, it suffices to show that $\tau \in \mathcal{H}_V(S)$. For that, note that $E_{-g}$ and $\alpha$ bound an annulus in the handlebody. Therefore, $\tau$ is an annulus twist and consequently in the handlebody group (compare \cite{Hensel}, Example 5.5).  
\end{proof}
Let $R' \subset R$ be the subring defined as the fixed point set of complex conjugation on $R.$ Since as discussed in Section \ref{1.2}, complex conjugation is a well defined automorphism on $R,$ the subring $R'$ is well defined and equal to the subring of real elements, i.e. $R' = R \cap \mathbb{R}.$\\[1ex] 
Consider the so called elementary transformations $$T_i(r'): x \mapsto x + r'\langle x,e_i \rangle e_i \ \text{ and } \ T_{i,j}(r): x \mapsto x + r\langle x, e_i \rangle e_j + \bar{r}\langle x, e_j \rangle e_i,$$ for $i,j = \pm 1, ... , \pm (g-1), |i| \neq |j|,  r' \in R', r \in R$. Here $\langle \cdot, \cdot \rangle$ stands for the intersection form on $R^{2g-2}$ discussed in Section \ref{1.2}. From \cite{Loo}, we have that all of them are in the image of $\Gamma_{S,C}^\#$. Of course, we can only expect the ones having an upper right block form to be in the image of $\Gamma^\#_{V,C}$. We prove that all of those indeed lie in the image.
\begin{lemma}\label{elementary matrices for handlebody group}
Let $r' \in R'$ and $r \in R$. Then the following statements hold:\\[1ex]
 1) \ $T_i(r')$ is in the image of $\rho$ if and only if $i > 0$.\\[1ex]
 2) \ $T_{i,j}(r)$ is in the image of $\rho$ if and only if at least one of the $i, j$ is positive.
\end{lemma} 
\begin{proof}
If in 1) $i \ngtr 0$ or in 2) both $i,j \ngtr 0$, then the elementary transformations would have a nonzero entry in the lower left block. So, the implication from left to right is clear.\\[1ex]
For the other implications, let $i, j \in \{\pm1, ... , \pm(g-1)\}$ with $i>0$, $|i| \neq |j|$. We have to show that $T_i(r')$ and $T_{i,j}(r)$ are in the image of $\rho$ for any $r' \in R', r \in R.$\\[1ex] 
Let $H \subset R^{2g-2}$ be defined as $H := \langle e_i, e_{-i} \rangle$. Let $A_H$ be the transformation that maps $H$ to $\langle e_1, e_{-1} \rangle$, i.e. defined by 
$$A_H(e_{\pm i}) := e_{\pm 1}, \, A_H(e_{\pm1}) := e_{\pm i} \text{ and } A_H(e_{\pm k}) := e_{\pm k} \text{ for all other } k.$$ 
Then $A_H \in \mathcal{P}.$ From Lemmas \ref{symplectic in image} and \ref{T}, we have that $T_H := A_H^{-1} \circ T \circ A_H$ is in the image of $\rho$. Note that $T_H$ is just multiplication with $\zeta$ on $H$ and the identity on the (with respect to the intersection form on $R^{2g-2}$) orthogonal subspace.\\[1ex] 
Similarly, let $H' := \langle e_i, e_{-i} + e_j \rangle$ and $A_{H'}$ be the transformation that maps $H'$ to $\langle e_1, e_{-1} \rangle$. If $|j| \neq 1,$ define $A_{H'}$ by  
$$A_{H'}(e_{\pm1}) := e_{\pm i}, \, A_{H'}(e_i) := e_1, \, A_{H'}(e_{-i}) := e_{-1} - e_j, \, A_{H'}(e_{-j}) := e_{-j} - e_1.$$
If $j = 1,$ define $A_{H'}$ by 
$$A_{H'}(e_{1}) := e_{i}, \, A_{H'}(e_i) := e_1, \, A_{H'}(e_{-1}) := e_{-i} - e_1, \, A_{H'}(e_{-i}) := e_{-1} - e_i.$$ 
If $j = -1,$ define $A_{H'}$ by 
$$A_{H'}(e_{1}) := e_1 + e_{i}, \, A_{H'}(e_i) := e_1, \, A_{H'}(e_{-1}) := e_{-i}, \, A_{H'}(e_{-i}) := e_{-1} - e_{-i},$$ 
and in any case $A_{H'}(e_k) := e_k$ for all other non specified $k$.\\[1ex]
Again, we see that $A_{H'} \in \mathcal{P}.$ Let $T_{H'} := A_{H'}^{-1} \circ T \circ A_{H'}.$ As before, $T_{H'}$ is just multiplication by $\zeta$ on $H'$ and lies in the image of $\rho.$\\[1ex]
Computing $T_{i,j}(1 - \zeta^k)$ and $T_H^{-k} \circ T_{H'}^k$ for every basis element by using the above formulas, we see that $T_{i,j}(1 - \zeta^k) = T_H^{-k} \circ T_{H'}^k \in \text{image}(\rho)$ for any $k \in \mathbb{Z}$.\\[1ex]
We perform an example computation for the basis element $e_{-i}$ in the case where $i, j > 0$ and $|j| \neq 1:$
$$T_H^{-k} \circ T_{H'}^{k}(e_{-i}) =T_H^{-k}(\zeta^ke_{-i} + (\zeta^k - 1)e_j) = e_{-i} - (1 - \zeta^k)e_j = T_{i,j}(1 - \zeta^k)(e_{-i}).$$ 
The other cases can be checked similarly.\\[1ex]
Since $T_{i,j}(1)$ is in $\mathcal{P}$, we conclude that 
$$T_{i,j}(\zeta^k) = T_{i,j}(1 - \zeta^k)^{-1} \circ T_{i,j}(1) \in \text{image}(\rho),$$ and consequently any $T_{i,j}(r),$ with $r \in R,$ lies in the image of $\rho.$\\[1ex]
Another computation shows $[T_{i,-j}(\zeta^k), T_{i,j}(1)] = T_i(\zeta^k+\zeta^{-k})$ for any $k \in \mathbb{Z}$. From that, we obtain that $T_i(\zeta^k+\zeta^{-k}) \in \text{image}(\rho).$\\[1ex]
Together with the fact that $T_i(n) \in \mathcal{P} \subset \text{image}(\rho)$ for all $n \in \mathbb{Z},$ we obtain that any $T_i(r'),$ with $r' \in R',$ is in the image. This proves 1) and the part of 2) with $i>0$.\\[1ex] 
In order to prove the part with $j >0$, just exchange the roles of $i$ and $j$ and argue analogously.
\end{proof}
With this lemma, we are finally able to prove Proposition \ref{prop}.
\begin{proof}
Let $$A = \begin{pmatrix}
Id & B\\
0 & Id
\end{pmatrix}$$ with $B = B^*.$ We have to show $A \in \text{image}(\rho).$
Let $b_{ij}$ be the entries of $B$, where $i$ and $j$ range from $1$ to $g-1$. Then $b_{ij} = \overline{b_{ji}}$ for all $i \neq j$ and $b_{ii} \in R'$.\\[1ex] 
Note that the elementary transformation $T_i(-b_{ii})$ corresponds to the matrix $$\begin{pmatrix}
Id & b_{ii}E_{ii}\\
0 & Id\\
\end{pmatrix}$$ and that $T_{ij}(-b_{ji})$ corresponds to $$\begin{pmatrix}
Id & b_{ji}E_{ji}+b_{ij}E_{ij}\\
0 & Id\\
\end{pmatrix},$$ where $E_{ij}$ stands for the matrix with a $1$ in the $(i,j)$-entry and zeros everywhere else.\\[1ex] 
Multiplying any two such matrices yields a matrix, where we still have the identity on the diagonal blocks and the sum of the two upper right blocks in the upper right block. Therefore, multiplying all $T_i(-b_{ii})$ for $i=1, ... , g-1$ and all $T_{ij}(-b_{ji})$ for $i \neq j$ in whichever order yields our matrix $A$. The claim now follows from Lemma \ref{elementary matrices for handlebody group}. 
\end{proof}

This discussion finishes the proof of Proposition \ref{prop} and therefore also the proof of Theorem \ref{main theorem} which for completeness we state again here.
\begin{theorem*}
The image of $\Gamma_{V,C}^\# \to GL_{2g-2}(R)$ (and that of $\Gamma_{V,C,x_0} \to GL_{2g-2}(R)$) is the subgroup 
$$\Lambda := \biggl\{\begin{pmatrix}
    (D^*)^{-1} & B\\
    0 & D\\
\end{pmatrix} \, \biggl\mid \, \det(D) = \pm \zeta^k, D^*B = B^*D\biggl\}.$$
\end{theorem*}
We end this section with a Corollary. Let $Q_{2g-2}(R) \subset U_{2g-2}(R)$ denote the group of upper right block matrices $A$ that satisfy $A^*\Omega A = \Omega.$ Each such matrix has a power of $\zeta$ as determinant. Let $Q_{2g-2}^\#(R) \subset Q_{2g-2}(R)$ denote the subgroup of matrices with determinant an even power of $\zeta.$ Then the inclusion $Q_{2g-2}^\#(R) \subset Q_{2g-2}(R)$ is an equality if $d = |C|$ is odd and is of index two if $d$ is even.\\[1ex]
So, $\Lambda = \text{image}(\rho)$ is the subgroup of matrices in $Q_{2g-2}^\#(R)$ such that the diagonal blocks have determinant $\pm \zeta^k.$ In general, the lower right block of a matrix in $Q_{2g-2}^\#(R)$ has a determinant of the form $\pm r'\zeta^k,$ where $r'$ is a real unit of $\mathbb{Z}[\zeta]$.\\[1ex]
In the following corollary, a group $G$ is said to \textit{virtually surject} onto a group $H,$ if there is a homomorphism $G \to H$ whose image is of finite index. Our theorem implies: 
\begin{corollary}\label{cor}
The punctured handlebody group $\mathcal{H}_V(S,x_0)$ virtually surjects onto the groups $Q_{2g-2}(\mathbb{Z}[\zeta_d])$, where $d \in \{2,3,4,6\}$ and $\zeta_d$ is a $d^{\text{th}}$ root of unity.\\[1ex]
The handlebody group $\mathcal{H}_V(S)$ virtually surjects onto the groups $Q_{2g-2}(\mathbb{Z}[\zeta_d])/\langle \zeta_d \rangle$, where again $d \in \{2,3,4,6\}, \zeta_d$ is a $d^{\text{th}}$ root of unity and $\langle \zeta_d \rangle$ denotes the cyclic subgroup generated by the matrix $$\begin{pmatrix}
\zeta_d & \cdots & 0  \\
\vdots & \ddots & \vdots\\
0 & \cdots & \zeta_d
\end{pmatrix}.$$
\end{corollary}
\begin{proof}
For $\mathcal{H}_V(S,x_0),$ we know from Theorem \ref{main theorem} that its finite index subgroups $\Gamma_{V,\mathbb{Z}/d\mathbb{Z},x_0}$ surjects onto $\Lambda.$\\[1ex]
Consider the chain of inclusions $\Lambda \subset Q_{2g-2}^\#(R) \subset Q_{2g-2}(R).$ The second inclusion is of finite index, the first in general is not. However, in the cases where $d \in \{2,3,4,6\},$ there are no real units in $\mathbb{Z}[\zeta_d]$ other than $\pm 1$ and we obtain $\Lambda = Q_{2g-2}^\#(R).$ This implies that $\Lambda$ is of finite index in $Q_{2g-2}(R)$ in these cases, which proves the claim in the punctured case.\\[1ex]
The claim for $\mathcal{H}_V(S)$ follows in an analogous way from Theorem \ref{main theorem} and the commutative diagram
\begin{center}
\begin{tikzcd}
\Gamma_{V,C}^\# \arrow{r} \arrow{d} & Q_{2g-2}(\mathbb{Z}[\zeta_d]) \arrow{d}\\
\Gamma_{V,C} \arrow{r} & Q_{2g-2}(\mathbb{Z}[\zeta_d])/\langle \zeta_d \rangle.  
\end{tikzcd}
\end{center}
\end{proof}

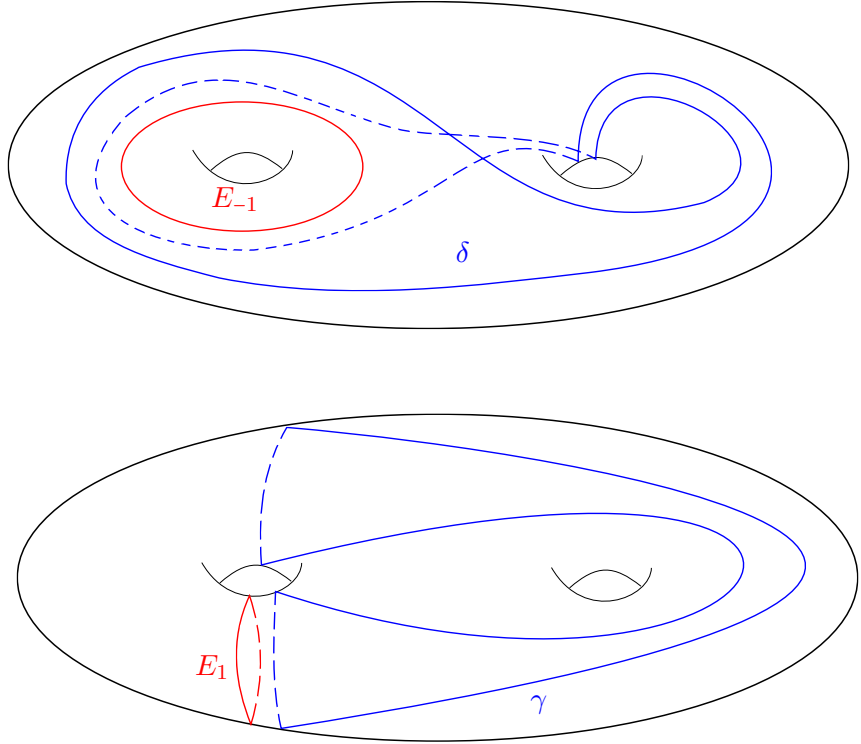
\begin{figure}[h]
\centering
\begin{tikzpicture}
%\draw[step=1cm,color=gray] (0,0) grid (11,10);%Uncomment this to get some helpful grid lines
\node[anchor=south west,inner sep=0] at (0,0){\includegraphics[scale=0.7]{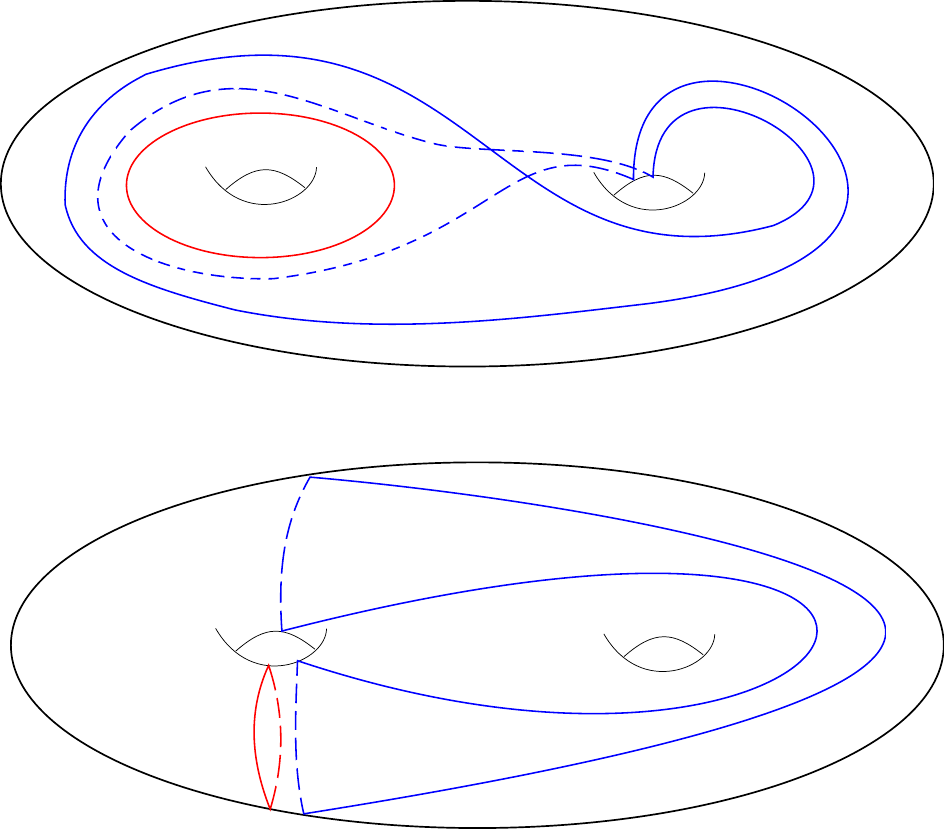}};
\node at (2.7,1) {$\textcolor{red}{E_1}$};
\node at (7,0.5) {$\textcolor{blue}{\gamma}$};
\node at (3,7.2) {$\textcolor{red}{E_{-1}}$};
\node at (6,6.5) {$\textcolor{blue}{\delta}$};
\end{tikzpicture}
\caption{The curves $\gamma$ and $\delta$}
\label{figure7}
\end{figure}

\begin{remark}
In the cases where $d \neq 2,3,4,6$ there are matrices in $Q_{2g-2}^\#(R)$ that are not in the image of $\rho.$ They are however in the image of the representation of $\Gamma_{S,C}^\#.$ An example (for the genus $2$ surface and the representation with $d = 5$) of a mapping class that maps to a matrix in $Q_{2}^\#(R) \setminus \Lambda$ is the mapping class 
$$T_{\gamma}^2 \circ T_{E_1}^{-2} \circ T_{E_{-1}} \circ T_\gamma^2 \circ T_{E_1}^{-6} \circ T_{\delta}^2 \circ T_{E_{-1}}^{-3},$$
where the curves that we use for the Dehn twists can be seen in figure \ref{figure7}.\\[1ex]
This mapping class has a lift mapping to the matrix $$\begin{pmatrix}
\sqrt{5} - 2 & 0\\
0 & \sqrt{5} + 2
\end{pmatrix}.$$ A concrete computation for this fact can be found in ("Representations of the Mapping Class Group" (Master's thesis, P. Bader), Example 5.1.3).\\[1ex] 
Here, we just point out that the Dehn twists about $\gamma$ and $\delta$ map to the matrices $$\begin{pmatrix}
1 & \zeta + \zeta^{-1} - 2\\
0 & 1\\
\end{pmatrix}$$ and 
$$\begin{pmatrix}
1 & 0\\
2 - \zeta - \zeta^{-1} & 1\\
\end{pmatrix}$$ respectively, where $\zeta$ in this case denotes a $5^{\text{th}}$ root of unity.\\[1ex]
Note that the constructed mapping class is not in the handlebody group, since it doesn't map all meridians to meridians.   

\end{remark}

\subsection{The genus 2 case}\label{genus 2}

Let $V$ be a genus $2$ handlebody, $C = \mathbb{Z}/d\mathbb{Z}, R = \mathbb{Z}[\zeta_d],$ where $\zeta_d$ is a $d^{\text{th}}$ root of unity and $\Gamma_{V,C}^\# \to GL_{2}(R)$ the corresponding representation.\\[1ex]
Let $R' := R \cap \mathbb{R}.$ Theorem \ref{main theorem} implies the following Corollary:

\begin{corollary}\label{cor2}
In the genus $2$ case the image of $\Gamma_{V,C}^\# \to GL_{2}(R)$ is the group of matrices $$\Lambda = \biggl\{\zeta_d^k
\begin{pmatrix}
\pm 1 & r'\\
0 & \pm 1
\end{pmatrix} \, \biggl\mid \, r' \in R', k \in \mathbb{Z}\biggl\}.$$
\end{corollary}

Note that for even $d,$ we could replace the $\pm 1$ in the above corollary by just $+1$ since $-1$ is a $d^{th}$ root of unity already. Corollary \ref{cor2} allows us to find virtual surjections of the genus $2$ handlebody group. Let $\mathbb{Z}[\zeta_d]'$ denote the (additive) subgroup of real elements of $\mathbb{Z}[\zeta_d]$.

\begin{corollary}
The genus $2$ handlebody group $\mathcal{H}_V(S)$ virtually surjects onto the groups $\mathbb{Z}/2\mathbb{Z} \oplus \mathbb{Z}[\zeta_d]'$ for any odd $d \in \mathbb{N}$ and onto $\mathbb{Z}[\zeta_d]'$ for any even $d \in \mathbb{N}.$ In particular, there exists a virtual surjection $\mathcal{H}_V(S) \dashedrightarrow \mathbb{Z}.$
\end{corollary}
\begin{proof}
Let $C = \mathbb{Z}/d\mathbb{Z}$ for some $d \in \mathbb{N}.$ The surjection $\Gamma_{V,C}^\# \to \Lambda$ induces (as in the proof of Corollary \ref{cor}) a surjection $$\Gamma_{V,C} \to \Lambda/\langle \zeta_d \rangle.$$ 
For $d$ even, we have 
$$\Lambda/\langle \zeta_d \rangle \cong \biggl\{\begin{pmatrix}
1 & r'\\
0 & 1
\end{pmatrix} \, \biggl\mid \, r' \in  \mathbb{Z}[\zeta_d]'\biggl\} \cong \mathbb{Z}[\zeta_d]'.$$
For $d$ odd on the other hand, we obtain
$$\Lambda/\langle \zeta_d \rangle \cong \biggl\{\begin{pmatrix}
\pm 1 & r'\\
0 & \pm 1
\end{pmatrix} \, \biggl\mid \, r' \in  \mathbb{Z}[\zeta_d]'\biggl\} =: \Theta.$$
The subgroups $$\{\pm Id\} \subset \Theta$$ and $$\biggl\{\begin{pmatrix}
1 & r'\\
0 & 1
\end{pmatrix} \, \biggl\mid \, r' \in  \mathbb{Z}[\zeta_d]'\biggl\} \subset \Theta$$ are both normal and generate $\Theta.$\\[1ex] 
Since these two subgroups are isomorphic to $\mathbb{Z}/2\mathbb{Z}$ and $\mathbb{Z}[\zeta_d]'$ respectively, we obtain $$\Theta \cong \mathbb{Z}/2\mathbb{Z} \oplus \mathbb{Z}[\zeta_d]'.$$ The fact that $\Gamma_{V,C}$ is a finite index subgroup of the handlebody group proves the claim.\\[1ex]
One can obtain the virtual surjection onto $\mathbb{Z}$ by choosing for example $d = 4,$ so that $\mathbb{Z}[\zeta_d]' = \mathbb{Z}[i]' = \mathbb{Z}.$
\end{proof}

In fact, the first homology of the genus $2$ handlebody group is $\mathbb{Z} \oplus \mathbb{Z}/2\mathbb{Z} \oplus \mathbb{Z}/2\mathbb{Z},$ so the whole genus $2$ handlebody group surjects onto $\mathbb{Z}.$ However, for $g \ge 3,$ the first homology of the genus $g$ handlebody group is $\mathbb{Z}/2\mathbb{Z},$ which means that it doesn't admit a surjection onto $\mathbb{Z}.$ For $g = 3$ it does admits a virtual surjection onto $\mathbb{Z},$ but for $g \ge 4$ it is still an open question whether there even exists a virtual surjection of the handlebody group onto the integers. See Theorem 8.5 and the subsequent discussion in \cite{Hensel} for more details.

\section{Image of the Twist Group}\label{4}

The purpose of this section is to determine the image of the twist group under our representations. This will be done explicitely by finding certain curves such that the twists about them map to matrices that generate the image of the twist group. We end this section by concluding that the twist group surjects onto the integers.\\[1ex]
The twist group $\mathcal{T}_V(S)$ is a subgroup of the handlebody group $\mathcal{H}_V(S)$. It can be defined as the kernel of the natural map $\mathcal{H}_V(S) \to \Out(\mathbb{F}_g)$ as discussed in Section \ref{1.1}. It is well known that it is generated by twists about meridians.\\[1ex]
For simplicity, we will write 
$\mathcal{T} = \mathcal{T}_V(S),$ since we've fixed the handlebody this group depends on. In this section, we will determine the image of $\mathcal{T}$ under our representation.\\[1ex]
More precisely: Let as always $C$ be a cyclic group of order $d,$ $R = \mathbb{Z}[\zeta]$ with $\zeta$ a $d^{\text{th}}$ root of unity and consider the representation $\Gamma_{V,C}^\# \to GL_{2g-2}(R),$ which has image $\Lambda$ (compare Theorem \ref{main theorem}). 
Let $\mathcal{T}^\#$ be the group of lifts of $\mathcal{T} \cap \Gamma_{V,C}.$ Then $\mathcal{T}^\#$ is a subgroup of $\Gamma_{V,C}^\#,$ so we can ask what its image under the above representation is.\\[1ex] 
In particular, we have a map 
$$\mathcal{T}^\# \to \Lambda$$ with $$\Lambda = \biggl\{\begin{pmatrix}
(D^*)^{-1} & B\\
0 & D\\ 
\end{pmatrix} \, \biggl\mid \, \det(D) = \pm \zeta^k, \, D^*B = B^*D\biggl\}$$ and we want to determine the image of this map.\\[1ex]
First of all, we want to show that $\mathcal{T}$ is completely contained in $\Gamma_{V,C},$ and consequently $\mathcal{T}^\#$ consists of the lifts of all elements in the twist group.
\begin{lemma}\label{twist lemma}
It holds that $\mathcal{T} \subset \Gamma_{V,C}.$
\end{lemma} 
\begin{proof}
It suffices to show that a generator of $\mathcal{T},$ i.e. a twist about a meridian, lies in $\Gamma_{V,C}.$ So, let $\alpha$ be a meridian and consider the Dehn twist $T_\alpha.$ We give an algebraic as well as a geometric proof.\\[1ex]
\textbf{Algebraic proof}: By the definition of $\Gamma_{V,C}$, we have to check that the action of $T_\alpha$ on homology preserves the subgroup $\ker(H_1(S) \to C)$ and induces the identity on $C.$ Recall that in our case the map $H_1(S) \to C \cong \mathbb{Z}/d\mathbb{Z}$ is given by $E_{-g} \mapsto 1 \, (\text{mod } d)$ and all other basis vectors map to $0.$ So, we have $\ker(H_1(S) \to C) = \langle E_{\pm1}, ... , E_{\pm (g-1)}, dE_{-g} \rangle.$ The action of $T_\alpha$ on homology is given by $x \mapsto x + (x, [\alpha]) [\alpha],$ where the homology class of $\alpha$ is of the form $[\alpha] = a_1E_1 + ... + a_gE_g,$ since it is a meridian (cf. Lemma \ref{meridians}). Now, one easily checks that the above claim is satisfied, and therefore we obtain $T_\alpha \in \Gamma_{V,C}.$\\[1ex]
\textbf{Geometric proof}: Since $\alpha$ is a meridian and our covering was constructed in a way that it also gives rise to a covering of handlebodies, we get that $\alpha$ lifts to $d$ homeomorphic copies $\alpha_1, ... , \alpha_d,$ which are all meridians as well. (This will be explained in more detail in the geomtric proof of Proposition \ref{bbb} below.)\\[1ex]
Therefore, $T_\alpha$ admits a lift of the form $T_{\alpha_1} \circ ... \circ T_{\alpha_d}.$ This shows that $T_\alpha$ does indeed lift to the covering. Now, let $c \in C$ be any deck transformation and note that $c \circ T_\beta = T_{c(\beta)} \circ c$ for any curve $\beta.$ Hence, we obtain 
$$c \circ T_{\alpha_1} \circ ... \circ T_{\alpha_d} = T_{c(\alpha_1)} \circ c \circ T_{\alpha_2} \circ ... \circ T_{\alpha_d} = T_{c(\alpha_1)} \circ ... \circ T_{c(\alpha_d)} \circ c = T_{\alpha_1} \circ ... \circ T_{\alpha_d} \circ c.$$ 
In the computation, we use that $c$ permutes the $\alpha_i,$ and since these are disjoint, we can rearrange the order we twist about them as we like. So, we also get that any lift of $T_\alpha$ commutes with the deck transformations, i.e. $T_\alpha \in \Gamma_{V,C}.$ 
\end{proof}
From the lemma, we obtain that $\mathcal{T}^\#$ is the group of lifts of all elements in the twist group. Therefore, we know that it is generated by the lifts of twists about meridians. We will use this in order to prove the following statement about the image of the representation.
\begin{proposition}\label{bbb}
The image of the representation $\mathcal{T}^\# \to \Lambda$ is contained in the subgroup
$$\Delta := \biggl\{ \zeta^k \begin{pmatrix}
Id & B\\
0 & Id\\ 
\end{pmatrix} \biggl\mid \ B = B^*, \ k \in \mathbb{Z} \biggl\}.$$
\end{proposition}
\begin{proof}
By the above lemma, it suffices to show that a twist about a meridian has a lift that maps to a matrix of the form $$\begin{pmatrix}
Id & B\\
0 & Id\\ 
\end{pmatrix},$$
because all other lifts just differ by a deck transformation and therefore map to a matrix of the form $$\zeta^k\begin{pmatrix}
Id & B\\
0 & Id\\ 
\end{pmatrix} \in \Delta.$$ 
We will again give an algebraic as well as a geometric proof of this claim.\\[1ex]
\textbf{Algebraic proof}: Recall from Section \ref{lrb} that we can describe the way to obtain the lower right block of our representation in a different way. Namely, for any $f \in \Gamma_{V,C}$, we can look at the induced element in $\Out(\mathbb{F}_g).$ A representative of this element in $\Aut(\mathbb{F}_g)$ induces an automorphism of $\pi_1(\widetilde{V}) = \ker(\mathbb{F}_g \to C),$ which in turn induces an automorphism of $H_1(\widetilde{V}; \mathbb{Q})$ and consequently one of $\mathbb{Z}[\zeta]^{g-1}.$ This automorphism corresponds to the lower right block matrix in our representation. The choice of a different representative in the step from $\Out(\mathbb{F}_g)$ to $\Aut(\mathbb{F}_g)$ leads to an element which differs from the previous by the action of a deck transformation, i.e. by multiplication with $\zeta^k$ for some $k$. Since the twist group is defined as the kernel of $\mathcal{H}_V(S) \to \Out(\mathbb{F}_g),$ this procedure shows that the lower right block of our representation will be the identity matrix (up to multiplication with $\zeta^k$) for any element in the twist group. In other words, any element in the twist group admits a lift, which maps to a matrix with $Id$ as its lower right (and upper left) block.\\[1ex]
\textbf{Geometric proof}: Let $\alpha$ be a meridian. Since the map $H_1(S) \to C$ corresponding to our covering factors through $H_1(V),$ the homology class of $\alpha$ maps to $0$ and therefore $\alpha$ lifts to the covering as a closed curve. Consequently the whole preimage of $\alpha$ consists of $d$ disjoint curves $\alpha_1, ... , \alpha_d$ which all map homeomorphically onto $\alpha$ under the covering map. Furthermore, since the covering is constructed in such a way that it is also a covering of handlebodies, the $\alpha_i$ are meridians. The composition of twists $T_{\alpha_1} \circ ... \circ T_{\alpha_d}$ is a lift of $T_\alpha.$  Since any two meridians have trivial algebraic intersection number (compare \cite{Hensel}, Lemma 2.1), any twist $T_{\alpha_i}$ fixes the homology class of any meridian. In particular, the homology classes of $e_1, ... , e_{g-1}$ are fixed under the action of the lift, i.e. we get that the upper left block (and therefore also the lower right block) is the identity.
\end{proof}
The above proposition says that the image of $\mathcal{T}^\#$ is contained in $\Delta$. We want to show the converse, i.e. that the image is all of $\Delta.$ For that, we have to understand how a lift of a Dehn twist acts on $R^{2g-2}.$\\[1ex]
Let $\alpha$ be a meridian and $T_{\alpha}$ the Dehn twist about $\alpha.$ Let $\alpha_1, ... , \alpha_d$ be the disjoint lifts of $\alpha.$ Then $\widetilde{T}_{\alpha} = T_{\alpha_1} \circ ... \circ T_{\alpha_d}$ is a lift of $T_{\alpha}.$ Let $\langle \cdot, \cdot \rangle$ denote the intersection form on $R^{2g-2}$ (compare Section \ref{1.2}).

\begin{lemma}\label{aaa}
The action of $\widetilde{T}_{\alpha}$ on $R^{2g-2}$ is given by $x \mapsto x + \langle x, \alpha_1 \rangle \alpha_1.$
\end{lemma}  
\begin{proof}
This can be checked using the definition of the intersection form on $R^{2g-2}$. We refer the reader to (\cite{Loo}, Section 3.1) for more details.
\end{proof}

\begin{theorem}
The image of the representation $\mathcal{T}^\# \to \Lambda$ is the subgroup $\Delta,$ where
$$\Delta = \biggl\{ \zeta^k \begin{pmatrix}
Id & B\\
0 & Id\\ 
\end{pmatrix} \, \biggl\mid \, B = B^*, \, k \in \mathbb{Z} \biggl\}.$$
\end{theorem}
\begin{proof}
We have to show that any self-adjoint $B$ can be an upper right block in the image of $\mathcal{T}^\#.$ By composing a lift that maps to such a desired matrix with all deck transformations, we get that all the 
$$\zeta^k \begin{pmatrix}
Id & B\\
0 & Id\\ 
\end{pmatrix}$$
are in the image of $\mathcal{T}^\#.$\\[1ex]
Note that the diagonal entries of $B$ are all in $R'$, while the off-diagonal entries are arbitrary elements in $R,$ where as always $R = \mathbb{Z}[\zeta]$ and $R'$ is the subring of real elements. Since $R$ is additively generated by the $\zeta^k$ and $R'$ by $1$ and the $\zeta^k + \zeta^{-k},$ it suffices to show that the matrices $E_{ii}, \, (\zeta^k+\zeta^{-k})E_{ii}$ and $\zeta^kE_{ji} + \zeta^{-k}E_{ij}$ occur as an upper right block in the image for any $i, j \in \{1, ... , g-1\}, i \neq j$. Here, $E_{mn}$ stands for the matrix with a $1$ in the (m,n)-entry and zeros otherwise.\\[1ex]
\textbf{The matrix $E_{ii}$}: This matrix is in the image as an upper right block, since a lift of the inverse of the twist about $E_i,$ which is a meridian, maps to it. This can be seen by the following calculation. The inverse twist about $E_i$ acts on $R^{2g-2}$ as
$$x \mapsto x - \langle x, e_i \rangle e_i.$$
From that, we obtain $e_{-j} \mapsto e_{-j}$ for all $j \neq i$ and $e_{-i} \mapsto e_{-i} + e_i$. In matrix notation this means that the upper right block is $E_{ii}$ as desired.\\[1ex]
\textbf{The matrix $(\zeta^k + \zeta^{-k})E_{ii}$}: Consider the two curves $E_i$ and $\bar{E}_i$ and the arc $\alpha_k$ (for $k \in \mathbb{N}$) as shown in figure \ref{newfigure} for $k = 1, 2, 3$. In general the arc $\alpha_k$ intersects $E_g$ exactly $k$ times. 

\begin{figure}[h]
\centering
\begin{tikzpicture}
%\draw[step=1cm,color=gray] (0,0) grid (14,6);%Uncomment this to get some helpful grid lines
\node[anchor=south west,inner sep=0] at (0,0){\includegraphics[scale=0.7]{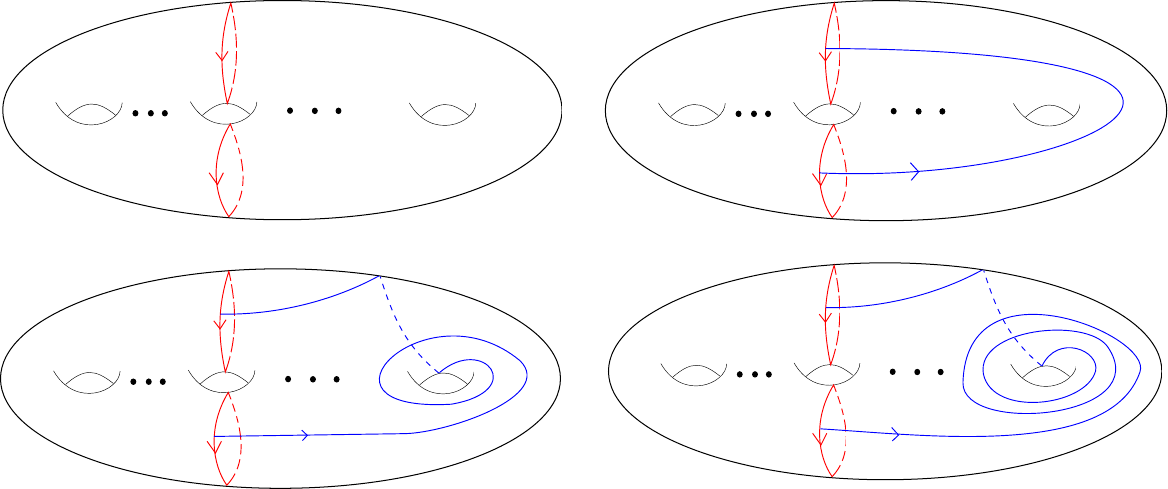}};
\node at (2.2,0.5) {$\textcolor{red}{E_i}$};
\node at (2.2,2) {$\textcolor{red}{\bar{E}_i}$};
\node at (2.2,3.7) {$\textcolor{red}{E_i}$};
\node at (2.2,5.1) {$\textcolor{red}{\bar{E}_i}$};
\node at (9.4,0.5) {$\textcolor{red}{E_i}$};
\node at (9.4,2) {$\textcolor{red}{\bar{E}_i}$};
\node at (9.4,3.7) {$\textcolor{red}{E_i}$};
\node at (9.4,5.1) {$\textcolor{red}{\bar{E}_i}$};
\node at (10.55,3.5) {$\textcolor{blue}{\alpha_1}$};
\node at (3.5,0.35) {$\textcolor{blue}{\alpha_2}$};
\node at (10.6,0.35) {$\textcolor{blue}{\alpha_3}$};
\end{tikzpicture}
\caption{The arcs $\alpha_k$ for $k = 1, 2, 3$}
\label{newfigure}
\end{figure}
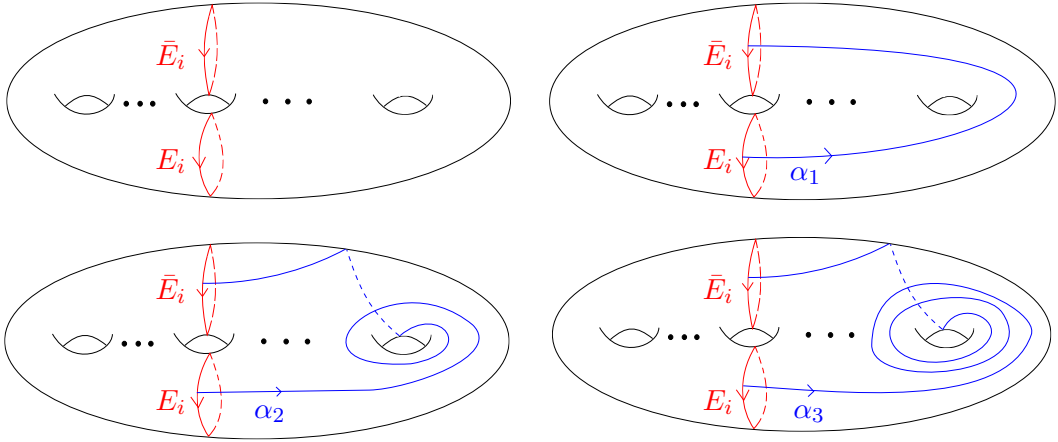

Let $D_k$ be a tubular neighbourhood of the arc $\alpha_k,$ i.e. $D_k \cong \alpha_k \times [0,1],$ on the surface $S.$ Then $D_k$ is a disk embedded in $S.$ We choose $D_k$ so that the curves $E_i, \bar{E}_i$ intersect $D_k$ only on its boundary. Now build the curve $\gamma = \gamma_{i,k}$ by taking the parts of $E_i$ and $\bar{E}_i$ that don't intersect $D_k$ and connect them via the parts of the boundary of $D_k$ that don't intersect $E_i$ and $\bar{E}_i.$ We obtain a curve as in figure \ref{figure8}.

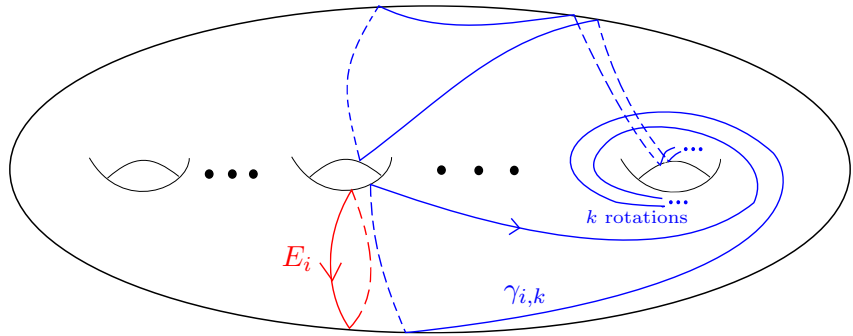
\begin{figure}[h]
\centering
\begin{tikzpicture}
%\draw[step=1cm,color=gray] (0,0) grid (11,4);%Uncomment this to get some helpful grid lines
\node[anchor=south west,inner sep=0] at (0,0){\includegraphics[scale=0.7]{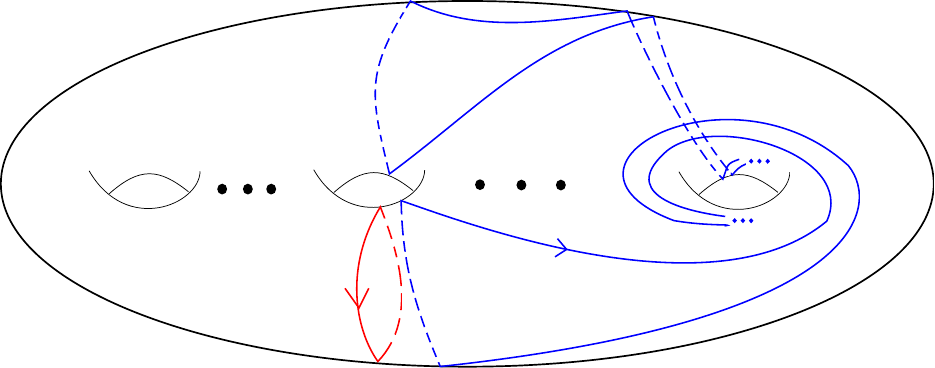}};
\node at (3.8,1) {$\textcolor{red}{E_i}$};
\node at (6.8,0.5) {$\textcolor{blue}{\gamma_{i,k}}$};
\node at (8.27,1.56) {$\textcolor{blue}{\text{\tiny{$k$ rotations}}}$};
\end{tikzpicture}
\caption{The curve $\gamma = \gamma_{i,k}$}
\label{figure8}
\end{figure}

Since $E_i$ and $\bar{E}_i$ bound disks $D$ and $\bar{D}$ respectively in $V, \gamma$ bounds the disk $D \cup D_k \cup \bar{D}.$ Hence $\gamma$ is a meridian and $T_\gamma \in \mathcal{T}.$ A lift of $\gamma$ can be seen in figure \ref{figure9}. Here, $c$ denotes a generator of the deck group that rotates the surface counterclockwise.

\begin{figure}[h]
\centering
\begin{tikzpicture}
%\draw[step=1cm,color=gray] (0,0) grid (9,7);%Uncomment this to get some helpful grid lines
\node[anchor=south west,inner sep=0] at (0,0){\includegraphics[scale=0.6]{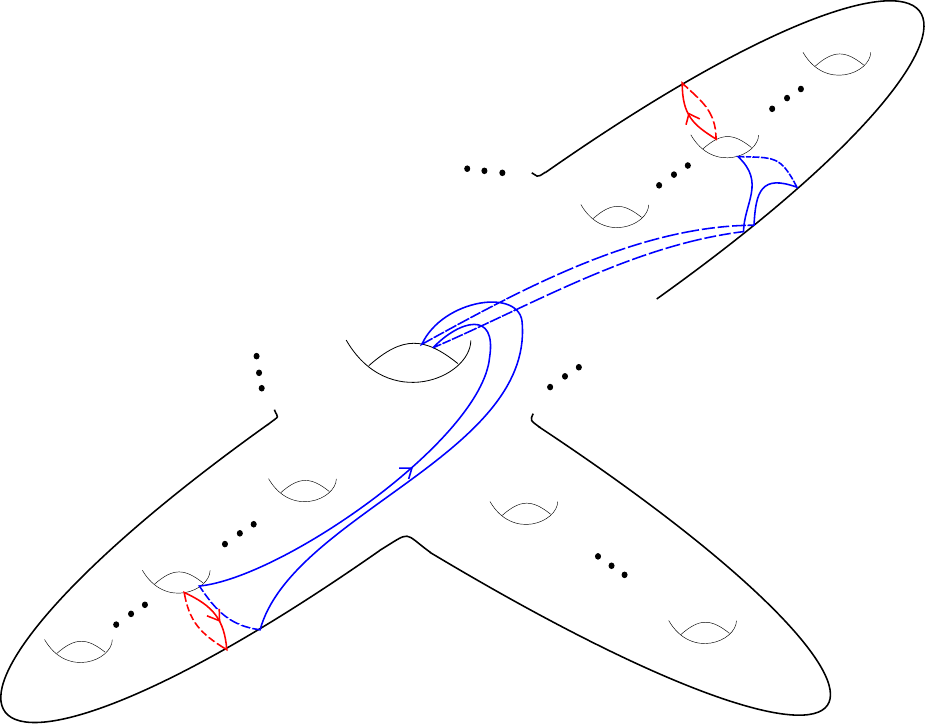}};
\node at (2.4,0.5) {$\textcolor{red}{e_i}$};
\node at (6.82,6.8) {$\textcolor{red}{c^ke_i}$};
\end{tikzpicture}
\caption{A lift of $\gamma = \gamma_{i,k}$}
\label{figure9}
\end{figure}

The $\mathbb{Z}[\zeta]$-valued homology class of this lift is $(1-\zeta^k)e_i.$ Therefore, from Lemma \ref{aaa} we know that a lift of $T_{\gamma}$ acts as 
$$x \mapsto x + \langle x, (1-\zeta^k)e_i \rangle (1-\zeta^k)e_i = x +  \overline{(1-\zeta^k)}(1-\zeta^k)\langle x,e_i \rangle e_i = x + (2- (\zeta^k + \zeta^{-k}))\langle x, e_i\rangle e_i.$$ 
In matrix notation, this has $(\zeta^k + \zeta^{-k} - 2)E_{ii}$ as upper right block. So, a lift of $T_\gamma \circ T^{-2}_{E_i}$ is the desired element mapping to a matrix with upper right block equal to $(\zeta^k + \zeta^{-k})E_{ii}.$\\[1ex]
\textbf{The matrix $\zeta^kE_{ji} + \zeta^{-k}E_{ij}$}: Assume, without loss of generality, that $i < j.$ Consider the following curve $\gamma = \gamma_{i,j,k}$ as in figure \ref{figure10}, which is build in the analogous way as $\gamma_{i,k}$ before by changing $\bar{E}_i$ to $\bar{E}_j$.

\begin{figure}[h]
\centering
\begin{tikzpicture}
%\draw[step=1cm,color=gray] (0,0) grid (11,4);%Uncomment this to get some helpful grid lines
\node[anchor=south west,inner sep=0] at (0,0){\includegraphics[scale=0.7]{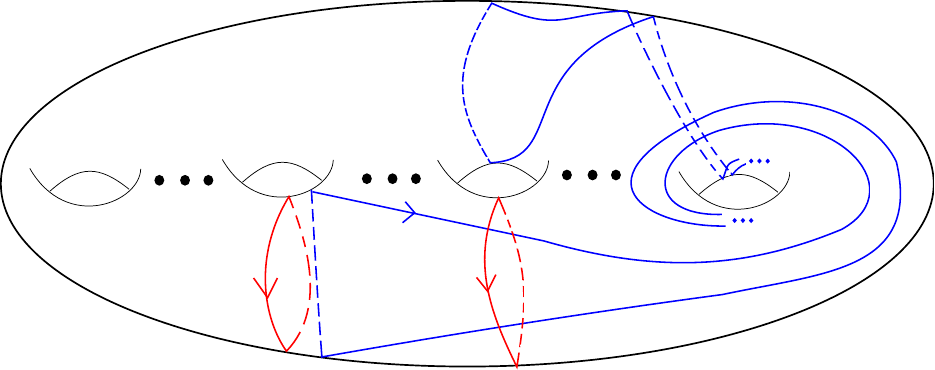}};
\node at (2.7,1.05) {$\textcolor{red}{E_i}$};
\node at (5.3,1) {$\textcolor{red}{E_j}$};
\node at (7.2,0.4) {$\textcolor{blue}{\gamma_{i,j,k}}$};
\node at (8.27,1.56) {$\textcolor{blue}{\text{\tiny{$k$ rotations}}}$};
\end{tikzpicture}
\caption{The curve $\gamma = \gamma_{i,j,k}$}
\label{figure10}
\end{figure}
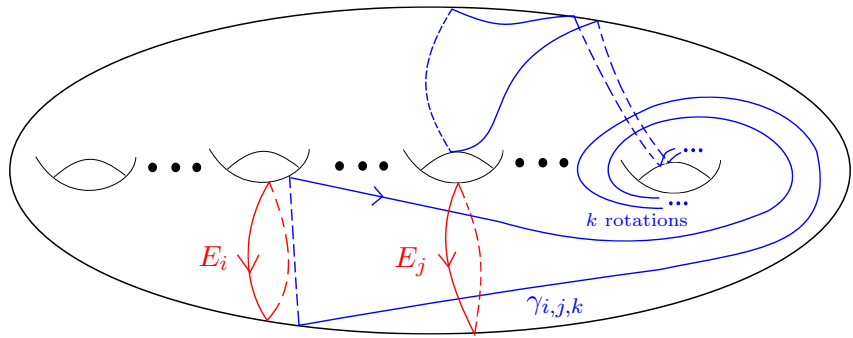

By the same reasoning as before, it is a meridian, so $T_{\gamma} \in \mathcal{T}.$ There is a lift of $\gamma$ with $\mathbb{Z}[\zeta]$-valued homology class equal to $e_i - \zeta^ke_j,$ as can be seen in figure \ref{figure11}. \vspace{1cm}

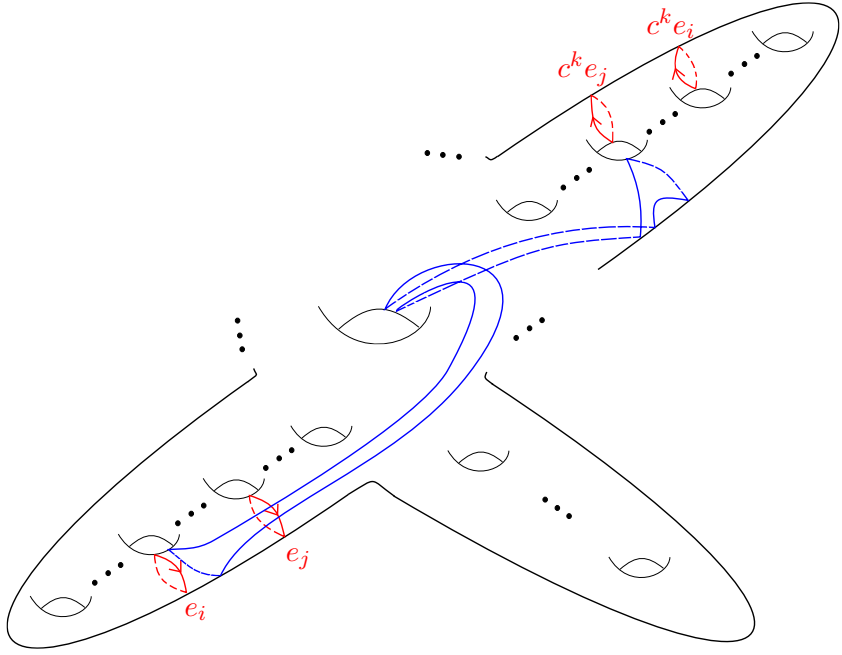
\begin{figure}[h]
\centering
\begin{tikzpicture}
%\draw[step=1cm,color=gray] (0,0) grid (11,9);%Uncomment this to get some helpful grid lines
\node[anchor=south west,inner sep=0] at (0,0){\includegraphics[scale=0.7]{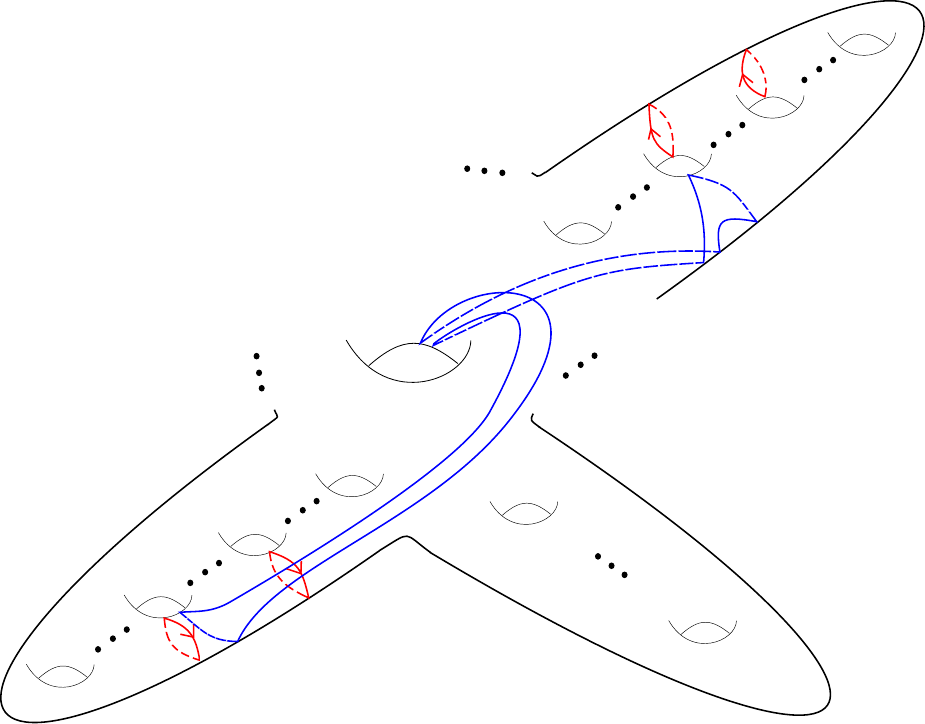}};
\node at (2.5,0.5) {$\textcolor{red}{e_i}$};
\node at (8.75,8.3) {$\textcolor{red}{c^ke_i}$};
\node at (3.85,1.2) {$\textcolor{red}{e_j}$};
\node at (7.6,7.7) {$\textcolor{red}{c^ke_j}$};
\end{tikzpicture}
\caption{A lift of $\gamma = \gamma_{i,j,k}$}
\label{figure11}
\end{figure}

So, there is a lift of $T_\gamma$, which acts as 
$$x \mapsto x + \langle x, e_i - \zeta^ke_j \rangle (e_i - \zeta^ke_j).$$
We compute the matrix notation of this map. Clearly, the upper left and lower right blocks are the identity and the lower left block is zero. The upper right block is zero everywhere, except for the $i.$th and $j.$th column. We compute:
$$e_{-i} \mapsto e_{-i} - e_i + \zeta^ke_j, \ \ \ \ \ e_{-j} \mapsto e_{-j} + \zeta^{-k}(e_i - \zeta^ke_j) = e_{-j} + \zeta^{-k}e_i - e_j.$$ 
Therefore, the upper right block is $$
\begin{pmatrix} 0 & 0 & 0\\
0 &
\begin{matrix}
-1 & \cdots & \zeta^{-k}\\
\vdots & \ddots & \vdots\\
\zeta^k & \cdots & -1
\end{matrix} & 0\\
0 & 0 & 0
\end{pmatrix},
$$ where all entries are zero, except for the four in the $(i,i), (i,j), (j,i)$ and $(j,j)$-position. In order to get rid of the two $-1$ in the $(i,i)$ and $(j,j)$-entry, we compose with the inverse of the twists about $E_i$ and $E_j.$ So, the mapping class $T_\gamma \circ T^{-1}_{E_i} \circ T^{-1}_{E_j}$ admits a lift that maps to a matrix with upper right block equal to $\zeta^kE_{ji} + \zeta^{-k}E_{ij}.$  
\end{proof}
By projection to the upper right block, we get the following corollary from the above theorem:
\begin{corollary}
The twist group $\mathcal{T}$ surjects onto the (additive) group of self-adjoint $(g-1) \times (g-1)$-matrices with entries in $R = \mathbb{Z}[\zeta_d]$ for any $d \in \mathbb{N}$ and $\zeta_d$ a $d^{\text{th}}$ root of unity. 
\end{corollary}
\begin{proof}
We have the commutative diagram
\begin{center}
\begin{tikzcd}
\mathcal{T}^\# \arrow{r} \arrow{d} & \Delta \arrow{d} \\
\mathcal{T} \arrow{r} & \Delta/\langle \zeta_d \rangle,
\end{tikzcd}
\end{center}
where all maps are surjections.\\[1ex] 
By $\langle \zeta_d \rangle,$ we mean the subgroup generated by the matrix $$\begin{pmatrix}
\zeta_d & &\\
 & \ddots & \\
 & & \zeta_d 
\end{pmatrix}.$$\\[1ex] 
Since $$\Delta/\langle \zeta_d \rangle \cong \biggl\{\begin{pmatrix}
Id & B\\
0 & Id\\ 
\end{pmatrix} \, \biggl\mid \, B = B^* \biggl\},$$ we get the claim by post-composing the bottom map of the diagram with the projection to the upper right block of the matrices.
\end{proof}
By post-composing the surjection given by the above corollary with the homomorphism projecting a matrix to one of its off-diagonal entries, we obtain: 
\begin{corollary}
The twist group (of any genus $\ge 3$ surface) surjects onto any $\mathbb{Z}[\zeta_d]$ with $d \in \mathbb{N}.$
\end{corollary}
By projecting to a diagonal entry, we also obtain:
\begin{corollary}
The twist group (of any genus $\ge 2$ surface) surjects onto the subgroup of real elements of $\mathbb{Z}[\zeta_d]$ for any $d \in \mathbb{N}.$
\end{corollary}
In particular, we obtain (for any genus $\ge 2$):
\begin{corollary}
There is a surjection $\mathcal{T} \twoheadrightarrow \mathbb{Z}$ of the twist group onto the integers.
\end{corollary}
The results about the surjections of the twist group also follow from (\cite{Twistgroup}, Theorem 1.2). There, it is shown that $\mathcal{T}$ surjects onto $\mathbb{Z}[\mathbb{Z}]$ by using an infinite cyclic covering instead of finite coverings.\\[1ex]

\textbf{Declarations.} This work was written while the author was supported by the UKRI Grant (Number: EP/V521917/1). The author has no relevant financial or non-financial interests to declare. 

\nocite{*}

\printbibliography

\end{document}